\documentclass[10pt]{amsart}

\usepackage[english]{babel}

\usepackage{amssymb}
\usepackage[leqno]{amsmath}
\usepackage{mathrsfs}
\usepackage{stmaryrd}
\usepackage{chemarrow}
\usepackage[norelsize]{algorithm2e}
\usepackage{enumerate}
\usepackage{graphicx}
\usepackage[all]{xy}
\usepackage{tikz}
\usetikzlibrary{arrows}
\usepackage{multirow}
\usepackage[colorinlistoftodos]{todonotes}
\usepackage[colorlinks=true, allcolors=blue]{hyperref}

\newtheorem{theorem}{Theorem}[section]
\newtheorem{lemma}[theorem]{Lemma}

\newtheorem{remark}[theorem]{Remark}

\newcommand{\bs }{\boldsymbol}

\newcommand{\sign}{\operatorname{sign}}
\newcommand{\curl}{\operatorname{curl}}
\renewcommand{\div}{\operatorname{div}}
\newcommand{\grad}{\operatorname{grad}}

\DeclareMathOperator*{\rot}{rot}

\newcommand{\sym}{\operatorname{sym}}

\newcommand{\defm}{\operatorname{def}}

\begin{document}
\title[Finite elements for divdiv-conforming symmetric tensors]{Finite elements for divdiv-conforming symmetric tensors}
\author{Long Chen}%
\address{Department of Mathematics, University of California at Irvine, Irvine, CA 92697, USA}%
\email{chenlong@math.uci.edu}%
\author{Xuehai Huang}%
\address{School of Mathematics, Shanghai University of Finance and Economics, Shanghai 200433, China}%
\email{huang.xuehai@sufe.edu.cn}%

\thanks{The first author was supported by NSF DMS-1913080.}
\thanks{The second author was supported by the National Natural Science Foundation of China Project 11771338 and the Fundamental Research Funds for the Central
Universities 2019110066.}


%
\begin{abstract}
Two types of finite element spaces on triangles are constructed for div-div conforming symmetric tensors. Besides the normal-normal continuity, the stress tensor is continuous at vertices and another trace involving combination of derivatives of stress is identified. Polynomial complex, finite element complex, and Hilbert complex are presented and a commuting diagram between them is given. The constructed div-div conforming elements are exploited to discretize the mixed formulation of the biharmonic equation. Optimal order and superconvergence error analysis is provided. By rotation, finite elements for rot-rot conforming symmetric strain are also obtained.
\end{abstract}
\maketitle


\section{Introduction}
In this paper, we shall construct finite element spaces of symmetric stress tensor conforming to the $\div \boldsymbol  \div$ operator and also present the Hilbert complex for finite element spaces.

Let $\Omega$ be a bounded polygon in $\mathbb R^2$. In \cite{ChenHuHuang2018}, we have found the Hilbert complexes and the commutative diagram
$$
\begin{array}{c}
\xymatrix{
  \boldsymbol{RT} \ar[r]^-{\subset} & \boldsymbol{H}^{1}(\Omega; \mathbb{R}^2) \ar[d]^{\boldsymbol{I}_h} \ar[r]^-{\sym\boldsymbol \curl}
                & \boldsymbol{H}^{-1}(\mathrm{div}\boldsymbol{\mathrm{div}},\Omega; \mathbb{S}) \ar[d]^{\boldsymbol{\Pi}_h}   \ar[r]^-{\mathrm{div}\boldsymbol{\mathrm{div}}} & \ar[d]^{Q_h}H^{-1}(\Omega) \ar[r]^{} & 0 \\
 \boldsymbol  {RT} \ar[r]^{\subset} & \mathcal{S}_h \ar[r]^{\sym\boldsymbol \curl}
                & \mathcal{V}_h   \ar[r]^{(\mathrm{div}\boldsymbol{\mathrm{div}})_h} &  \mathcal{P}_h \ar[r]^{}& 0    }
\end{array},
$$
where $\boldsymbol{RT}$ is the lowest order Raviart-Thomas element on $\Omega$ \cite{RaviartThomas1977}, $\mathcal S_h$ is the vector Lagrange element of degree $k+1$, $\mathcal V_h$ is the Hellan-Herrmann-Johnson (HHJ)  (cf.~\cite{Hellan1967, Herrmann1967, Johnson1973}) element of degree $k$, and $\mathcal P_h$ is the scalar Lagrange space of degree $k+1$ based on a shape regular triangulation $\mathcal T_h$ of $\Omega$. Details on the spaces and interpolation operators can be found in \cite[Section 2.2]{ChenHuHuang2018}.
The negative Sobolev space
\[
\boldsymbol{H}^{-1}(\div \boldsymbol{\div },\Omega; \mathbb{S}):=\{\boldsymbol{\tau}\in \boldsymbol{L}^{2}(\Omega; \mathbb{S}): \div \mathbf{div}\boldsymbol{\tau}\in H^{-1}(\Omega)\}
\]
with squared norm $\|\boldsymbol{\tau}\|_{\boldsymbol{H}^{-1}(\div \boldsymbol{\div })}^2:=\|\boldsymbol{\tau}\|_{0}^2+\|\div \boldsymbol{\div }\boldsymbol{\tau}\|_{-1}^2$, where $\mathbb{S}$ is the space of all symmetric $2\times2$ tensor. We refer to \cite{Sinwel2009,PechsteinSchoberl2011,PechsteinSchoeberl2018,KrendlRafetsederZulehner2016,ChenHuang2018,RafetsederZulehner2018} for details on the space $\boldsymbol{H}^{-1}(\div \boldsymbol{\div },\Omega; \mathbb{S})$.
It is very difficult to construct $\boldsymbol H^{-1}(\div\div)$-conforming but $\boldsymbol H(\div)$-nonconforming symmetric tensor, cf. \cite[Remark 3.33]{Sinwel2009}, \cite[Remark 2.1]{PechsteinSchoberl2011} and \cite[page 108]{PechsteinSchoeberl2018}.
Alternatively, with the help of the physical quantities normal bending moment, twisting moment and effective transverse shear force, hybridized mixed methods can be employed to discretize the mixed formulation of the Kirchhoff-Love plate bending problem, such as the hybridized discontinuous Galerkin method in \cite{HuangHuang2019} and the discontinuous Petrov-Galerkin method in \cite{Fuhrer;Heuer:2019discrete}.

A more natural Sobolev space for $\div\boldsymbol{\div}$ operator is
\[
\boldsymbol{H}(\div \boldsymbol{\div },\Omega; \mathbb{S}):=\{\boldsymbol{\tau}\in \boldsymbol{L}^{2}(\Omega; \mathbb{S}): \div \mathbf{div}\boldsymbol{\tau}\in L^{2}(\Omega)\}
\]
with squared norm $\|\boldsymbol{\tau}\|_{\boldsymbol{H}(\div \boldsymbol{\div })}^2:=\|\boldsymbol{\tau}\|_{0}^2+\|\div \boldsymbol{\div }\boldsymbol{\tau}\|_{0}^2$.
The corresponding Hilbert complex is
\begin{equation}\label{eq:divdivcomplexL2-intro}
\resizebox{.9\hsize}{!}{$
\boldsymbol  {RT}\autorightarrow{$\subset$}{} \boldsymbol  H^1(\Omega;\mathbb R^2)\autorightarrow{$\sym\boldsymbol \curl$}{} \boldsymbol{H}(\div\boldsymbol{\div},\Omega; \mathbb{S}) \autorightarrow{$\div \boldsymbol{\div }$}{} L^2(\Omega)\autorightarrow{}{}0$}.
\end{equation}

Conforming finite element spaces for $\boldsymbol  H^1(\Omega;\mathbb R^2)$ are relatively easy to construct and the natural finite element space for $L^2(\Omega)$ is discontinuous polynomial spaces. The following question arises quite naturally: can we construct conforming finite element spaces for $\boldsymbol{H}(\div \boldsymbol{\div },\Omega; \mathbb{S})$ such that the complex \eqref{eq:divdivcomplexL2-intro} is preserved in the discrete case? The purpose of this paper is to give a positive answer. We will answer this question by constructing two types of finite element spaces on triangles which resembles the RT and BDM spaces for $\boldsymbol  H(\boldsymbol {\div},\Omega)$ \cite{BoffiBrezziFortin2013} in the vector case. Furthermore we shall construct the following commuting diagram
$$
\begin{array}{c}
\xymatrix{
  \boldsymbol{RT} \ar[r]^-{\subset} & \boldsymbol  H^1(\Omega;\mathbb R^2) \ar[d]^{\boldsymbol{I}_h} \ar[r]^-{\sym\boldsymbol \curl}
                & \boldsymbol{H}(\div \boldsymbol  \div, \Omega; \mathbb{S}) \ar[d]^{\boldsymbol{\Pi}_h}   \ar[r]^-{\div\boldsymbol{\div}} & \ar[d]^{Q_{h}}L^2(\Omega) \ar[r]^{} & 0 \\
 \boldsymbol{RT} \ar[r]^-{\subset} & \boldsymbol V_{\ell+1} \ar[r]^{\sym\boldsymbol \curl}
                &  \boldsymbol \Sigma_{\ell,k}   \ar[r]^{\div\boldsymbol{\div}} &  \mathcal Q_{h} \ar[r]^{}& 0    }
\end{array},
$$
where the domain of the interpolation operators are smoother subspaces of the spaces in the top complex. Details of spaces and operators can be found in Sections 2 and 3, respectively.

We give a glimpse on the conforming space of stress. Let $K$ be a triangle. The set of edges of $K$ is denoted by $\mathcal E(K)$ and the vertices by $\mathcal V(K)$. The shape function space is simply $\mathbb P_k(K;\mathbb S)$. The degree of freedom are given by
\begin{align*}
\boldsymbol \tau (\delta) & \quad\forall~\delta\in \mathcal V(K),\\
(\boldsymbol  n^{\intercal}\boldsymbol \tau\boldsymbol  n, q)_e & \quad\forall~q\in\mathbb P_{k-2}(e),  e\in\mathcal E(K),\\
(\partial_{t}(\boldsymbol  t^{\intercal}\boldsymbol \tau\boldsymbol  n)+\boldsymbol  n^{\intercal}\boldsymbol \div\boldsymbol \tau, q)_e & \quad\forall~q\in\mathbb P_{k-1}(e),  e\in\mathcal E(K),\\
(\boldsymbol \tau, \boldsymbol \varsigma)_K & \quad\forall~\boldsymbol \varsigma\in\nabla^2\mathbb P_{k-2}(K)\oplus\sym (\bs x^{\perp}\otimes \mathbb P_{k-2}(K;\mathbb R^2)).
\end{align*}

By rotation, we obtain conforming finite elements for $\boldsymbol{H}(\rot \boldsymbol{\rot },\Omega; \mathbb{S})$ with the shape functions being polynomials.
The conforming finite element strain complex and the corresponding commutative diagram are also constructed.
Some lower-order $\boldsymbol H(\rot\boldsymbol\rot)$-conforming finite elements are advanced in \cite{ChristiansenHu2019}, whose shape functions are piecewise polynomials based on the Clough-Tocher split of the triangle.

Then the $\boldsymbol H(\div\boldsymbol\div)$-conforming finite elements are exploited to discretize the mixed formulation of the biharmonic equation.
The discrete inf-sup condition follows from the commutative diagram for the $\div$-$\div$ complex, and we derive the optimal convergence of the mixed finite element methods. Furthermore, the discrete inf-sup condition based on mesh-dependent norms is established, by which we acquire a third or fourth order higher superconvergence of $|Q_hu-u_h|_{2,h}$ than the optimal one. With the help of this superconvergence, a new superconvergent discrete deflection is devised by postprocessing. Hybridization is also provided for the easy of implementation.

The rest of this paper is organized as follows. In Section 2, we present the Green's identity for div-div operator and analyze the trace of the Sobolev space $\boldsymbol{H}(\div \boldsymbol  \div, \Omega; \mathbb{S})$. In Section 3, the conforming finite elements for $\boldsymbol{H}(\div \boldsymbol{\div },\Omega; \mathbb{S})$ and $\boldsymbol{H}(\rot\boldsymbol{\rot },\Omega; \mathbb{S})$, the finite element div-div complex and the finite element strain complex are constructed.
Mixed finite element methods for the biharmonic equation are developed and analyzed in Section 4.


\section{Div-div Conforming Symmetric Tensor Space}
In this section, we shall study the Sobolev space $\boldsymbol{H}(\div \boldsymbol  \div, \Omega; \mathbb{S})$ for div-div operator. We first present a Green's identity based on which we can characterize the trace of $ \boldsymbol{H}(\div \boldsymbol  \div, \Omega; \mathbb{S}) $ on polygons and give a sufficient continuity condition for a piecewise smooth function to be in $ \boldsymbol{H}(\div \boldsymbol  \div, \Omega; \mathbb{S}) $.

\subsection{Notation}
Denote the space of all  $2\times2$ matrix by $\mathbb{M}$, all symmetric $2\times2$ matrix by $\mathbb{S}$, and all skew-symmetric $2\times2$ matrix by $\mathbb{K}$.  Given a bounded domain $G\subset\mathbb{R}^{2}$ and a
non-negative integer $m$, let $H^m(G)$ be the usual Sobolev space of functions
on $G$, and $\boldsymbol{H}^m(G; \mathbb{X})$ be the usual Sobolev space of functions taking values in the finite-dimensional vector space $\mathbb{X}$ for $\mathbb{X}$ being $\mathbb M$, $\mathbb{S}$, $\mathbb{K}$ or $\mathbb{R}^2$. The corresponding norm and semi-norm are denoted respectively by
$\Vert\cdot\Vert_{m,G}$ and $|\cdot|_{m,G}$.  If $G$ is $\Omega$, we abbreviate
them by $\Vert\cdot\Vert_{m}$ and $|\cdot|_{m}$,
respectively. Let $H_{0}^{m}(G)$ be the closure of $C_{0}^{\infty}(G)$ with
respect to the norm $\Vert\cdot\Vert_{m,G}$.  $\mathbb P_m(G)$ stands for the set of all polynomials in $G$ with the total degree no more than $m$, and $\mathbb P_m(G; \mathbb{X})$ denotes the tensor or vector version.
Let $Q_m^G$ be the $L^2$-orthogonal projection operator onto $\mathbb P_m(G)$.
For $G$ being a polygon, denote by $\mathcal{E}(G)$ the set of all edges of $G$, $\mathcal{E}^i(G)$ the set of all
interior edges of $G$ and $\mathcal{V}(G)$ the set of all vertices of $G$.

Let $\{\mathcal {T}_h\}_{h>0}$ be a regular family of polygonal meshes
of $\Omega$. Our finite element spaces are constructed for triangles but some results, e.g., traces and Green's formulae etc, hold for general polygons.   For each element $K\in\mathcal{T}_h$, denote by $\boldsymbol{n}_K = (n_1, n_2)^{\intercal}$ the
unit outward normal to $\partial K$ and write $\boldsymbol{t}_K:= (t_1, t_2)^{\intercal} = (-n_2, n_1)^{\intercal}$, a unit vector
tangent to $\partial K$. Without causing any confusion, we will abbreviate $\boldsymbol{n}_K$ and $\boldsymbol{t}_K$ as $\boldsymbol{n}$ and $\boldsymbol{t}$ respectively for simplicity.
Let $\mathcal{E}_h$, $\mathcal{E}^i_h$, $\mathcal{V}_h$ and $\mathcal{V}^i_h$ be the union of all edges, interior edges, vertices and interior vertices
of the partition $\mathcal {T}_h$, respectively.
For any $e\in\mathcal{E}_h$,
fix a unit normal vector $\boldsymbol{n}_e:= (n_1, n_2)^{\intercal}$ and a unit tangent vector $\boldsymbol{t}_e := (-n_2, n_1)^{\intercal}$.

For a column vector function $\boldsymbol \phi = (\phi_1, \phi_2)^{\intercal}$, differential operators for scalar functions will be applied row-wise to produce a matrix function. Similarly for a matrix function, differential operators for vector functions are applied row-wise. For a scalar function $\phi$, $\curl \phi:= (\partial_{x_2} \phi, -\partial_{x_1} \phi)^{\intercal}$ with $\boldsymbol  x=(x_1, x_2)^{\intercal}$ and for a vector function $\boldsymbol  v$, $\boldsymbol  \curl$ is applied row-wise and the result $\boldsymbol  \curl \, \boldsymbol  v$ is a matrix, whose symmetric part is denoted by $\sym\boldsymbol \curl$. That is $\sym\boldsymbol \curl \, \boldsymbol  v := (\boldsymbol  \curl \, \boldsymbol  v + (\boldsymbol  \curl \, \boldsymbol  v)^{\intercal})/2$.

\subsection{Green's identity}
We start from the Green's identity for smooth functions but on polygons.
\begin{lemma} [Green's identity]\label{lm:Green}
Let $K$ be a polygon, and let $\boldsymbol  \tau\in \mathcal C^2(K; \mathbb S)$ and $v\in H^2(K)$. Then we have
\begin{align}
(\div\boldsymbol \div\boldsymbol \tau, v)_K&=(\boldsymbol \tau, \nabla^2v)_K -\sum_{e\in\mathcal E(K)}\sum_{\delta\in\partial e}\sign_{e,\delta}(\boldsymbol  t^{\intercal}\boldsymbol \tau\boldsymbol  n)(\delta)v(\delta) \notag\\
&\quad - \sum_{e\in\mathcal E(K)}\left[(\boldsymbol  n^{\intercal}\boldsymbol \tau\boldsymbol  n, \partial_n v)_{e}-(\partial_{t}(\boldsymbol  t^{\intercal}\boldsymbol \tau\boldsymbol  n)+\boldsymbol  n^{\intercal}\boldsymbol \div\boldsymbol \tau,  v)_{e}\right], \label{eq:greenidentitydivdiv}
\end{align}
where
\[
\sign_{e,\delta}:=\begin{cases}
1, & \textrm{ if } \delta \textrm{ is the end point of } e, \\
-1, & \textrm{ if } \delta \textrm{ is the start point of } e.
\end{cases}
\]
\end{lemma}
\begin{proof}
We start from the standard integration by parts
\begin{align*}
\begin{aligned}
(\operatorname{div} \boldsymbol  \div \boldsymbol  \tau, v)_{K} &=-(\boldsymbol  \div  \boldsymbol   \tau, \nabla v)_{K}+\sum_{e \in \mathcal E(K)}(\boldsymbol  n^{\intercal}\boldsymbol  \div  \boldsymbol   \tau, v)_e \\
&=\left(\boldsymbol  \tau, \nabla^{2} v\right)_{K}-\sum_{e \in \mathcal E(K)} (\boldsymbol  \tau \boldsymbol  n, \nabla v)_e+\sum_{e \in \mathcal E(K)}(\boldsymbol  n^{\intercal}\boldsymbol  \div  \boldsymbol   \tau, v)_e.
\end{aligned}
\end{align*}
Now we expand $(\boldsymbol  \tau \boldsymbol  n, \nabla v)_e=(\boldsymbol  n^{\intercal} \boldsymbol  \tau \boldsymbol  n,  \partial_{n} v)_e+ (\boldsymbol  t^{\intercal} \boldsymbol  \tau \boldsymbol  n, \partial_{t} v)_e$ and apply integration by parts on each edge to the second term
$$
(\boldsymbol  t^{\intercal} \boldsymbol  \tau \boldsymbol  n, \partial_{t} v)_e=\sum_{\delta \in \partial e} \sign_{e,\delta}(\boldsymbol  t^{\intercal}\boldsymbol \tau\boldsymbol  n)(\delta)v(\delta)-(\partial_ t\left( \boldsymbol  t^{\intercal} \boldsymbol  \tau \boldsymbol  n\right), v)_e
$$
to finish the proof.
\end{proof}

In the context of elastic mechanics~\cite{FengShi1996}, $\boldsymbol  n^{\intercal}\boldsymbol \tau\boldsymbol  n$ and $\partial_{t}(\boldsymbol  t^{\intercal}\boldsymbol \tau\boldsymbol  n)+\boldsymbol  n^{\intercal}\boldsymbol \div\boldsymbol \tau$ are called
normal bending moment and effective transverse shear force respectively for $\bs\tau$ being a moment.

For a scalar function $\phi$, due to the rotation relation, $\boldsymbol  n^{\intercal} \curl \phi = \boldsymbol  t^{\intercal} \grad \phi =  \partial_t \phi$ and $\boldsymbol  t^{\intercal} \curl \phi = -\boldsymbol  n^{\intercal} \grad \phi = -\partial_n \phi$. For vector and matrix functions, we have the following relations.
\begin{lemma}\label{lm:tauv}
When $\boldsymbol  \tau = \sym \boldsymbol  \curl \, \boldsymbol  v$, we have the following identities
\begin{align}
\label{eq:trace1} \boldsymbol  n^{\intercal}\boldsymbol \tau\boldsymbol  n &= \boldsymbol  n^{\intercal} \partial_t \boldsymbol  v,\\
\label{eq:trace2} \partial_{t}(\boldsymbol  t^{\intercal}\boldsymbol \tau\boldsymbol  n)+\boldsymbol  n^{\intercal}\boldsymbol \div\boldsymbol \tau & =  \partial_t(\boldsymbol  t^{\intercal}\partial_t\boldsymbol  v).
\end{align}
\end{lemma}
\begin{proof}
The first one is a straight forward calculation using $(\boldsymbol  \curl \, \boldsymbol  v)\boldsymbol n =  \partial_t \boldsymbol  v$. We now focus on the second one.
Since $\boldsymbol  \div \boldsymbol  \curl \, \boldsymbol  v = 0$, we have
$$\boldsymbol  n^{\intercal}\boldsymbol \div\boldsymbol \tau  = \frac{1}{2}\boldsymbol  n^{\intercal} \boldsymbol  \div (\boldsymbol  \curl \, \boldsymbol  v)^{\intercal} = \frac{1}{2}\boldsymbol  n^{\intercal} \curl \div \boldsymbol  v =  \frac{1}{2} \partial_t \div \boldsymbol  v.$$
As $\div \boldsymbol  v = {\rm trace} (\boldsymbol  \nabla \boldsymbol  v)$ is invariant to the rotation, we can write it as
$$
\div \boldsymbol  v = \boldsymbol  t^{\intercal}\boldsymbol  \nabla \boldsymbol  v \boldsymbol  t + \boldsymbol  n^{\intercal}\boldsymbol  \nabla \boldsymbol  v \boldsymbol  n =  \boldsymbol  t^{\intercal}\partial_t\boldsymbol  v + \boldsymbol  n^{\intercal}\partial_n\boldsymbol  v.
$$
Then
$$
 \partial_{t}(\boldsymbol  t^{\intercal}\boldsymbol \tau\boldsymbol  n)+\boldsymbol  n^{\intercal}\boldsymbol \div\boldsymbol \tau  =\frac{1}{2}\partial_t[  \boldsymbol  t^{\intercal}\partial_t\boldsymbol  v -  \boldsymbol  n^{\intercal}\partial_n\boldsymbol  v + \div\boldsymbol  v]  =  \partial_t  (\boldsymbol  t^{\intercal}\partial_t\boldsymbol  v),
 $$
 i.e. \eqref{eq:trace2} holds.
\end{proof}

\subsection{Traces}
Next we recall the trace of the space $\boldsymbol{H}(\div \boldsymbol{\div },K; \mathbb{S})$ on the boundary of polygon $K$. {Detailed proofs of the following trace operators can be found in \cite[Theorem 2.2]{Amara;Capatina-Papaghiuc;Chatti:2002Bending} for 2-D domains and \cite[Lemma 3.2]{Fuhrer;Heuer;Niemi:2019ultraweak} for both 2-D and 3-D domains.} The normal-normal trace of $\boldsymbol{H}(\div \boldsymbol{\div },K; \mathbb{S})$ can be also found in \cite{Sinwel2009,PechsteinSchoeberl2018}.

Define trace space
\begin{align*}
H_{n,0}^{1/2}(\partial K)&:=\{\partial_n v|_{\partial K}: v\in H^2(K)\cap H_0^1(K)\} \\
&\;=\{g\in L^2(\partial K): g|_e\in H_{00}^{1/2}(e)\;\;\forall~e\in\mathcal E(K)\}
\end{align*}
with norm
\[
\|g\|_{H_{n,0}^{1/2}(\partial K)}:=\inf_{v\in H^2(K)\cap H_0^1(K)\atop \partial_n v=g}\|v\|_2.
\]
Let $H_n^{-1/2}(\partial K):=(H_{n,0}^{1/2}(\partial K))'$. Note that for a 2D polygon $K$, and $v\in H^2(K)\cap H_0^1(K)$, the normal derivative $\partial_n v|_{e}\in H_{00}^{1/2}(e)$ for boundary edge $e\in \partial K$ can be derived from the compatible condition for traces on polygonal domains \cite[Theorem 1.5.2.8]{Grisvard:2011Elliptic}.

\begin{lemma}
For any $\boldsymbol \tau\in\boldsymbol{H}(\div \boldsymbol{\div },K; \mathbb{S})$,  it holds
\[
\|\boldsymbol  n^{\intercal}\boldsymbol \tau\boldsymbol  n\|_{H_n^{-1/2}(\partial K)}\lesssim \|\boldsymbol{\tau}\|_{\boldsymbol{H}(\div \boldsymbol{\div })}.
\]
Conversely, for any $g\in H_n^{-1/2}(\partial K)$, there exists some $\boldsymbol \tau\in\boldsymbol{H}(\div \boldsymbol{\div },K; \mathbb{S})$ such that
\[
\boldsymbol  n^{\intercal}\boldsymbol \tau\boldsymbol  n|_{\partial K}=g, \quad
\|\boldsymbol{\tau}\|_{\boldsymbol{H}(\div \boldsymbol{\div })} \lesssim \|g\|_{H_n^{-1/2}(\partial K)}.
\]
The hidden constants depend only the shape of the domain $K$.
\end{lemma}

We then consider another part of the trace involving combination of derivatives. Define trace space
\begin{align*}
H_{e,0}^{3/2}(\partial K)&:=\{v|_{\partial K}: v\in H^2(K), \partial_nv|_{\partial K}=0, v(\delta)=0 \textrm{ for each vertex } \delta\in\mathcal V(K)\}
\end{align*}
with norm
\[
\|g\|_{H_{e,0}^{3/2}(\partial K)}:=\inf_{v\in H^2(K)\atop \partial_n v=0, v=g}\|v\|_2.
\]
Let $H_e^{-3/2}(\partial K):=(H_{e,0}^{3/2}(\partial K))'$. Note that since we consider polygon domains, we explicitly impose the condition $v(\delta) = 0$ for each vertex of the polygon. 


\begin{lemma}
For any $\boldsymbol \tau\in\boldsymbol{H}(\div \boldsymbol{\div },K; \mathbb{S})$,  it holds
\begin{equation}\label{eq:divdivtracee0}
\|\partial_t(\boldsymbol  t^{\intercal}\boldsymbol \tau\boldsymbol  n)+\boldsymbol  n^{\intercal}\boldsymbol \div\boldsymbol \tau\|_{H_e^{-3/2}(\partial K)}\lesssim \|\boldsymbol{\tau}\|_{\boldsymbol{H}(\div \boldsymbol{\div })}.
\end{equation}
Conversely, for any $g\in H_e^{-3/2}(\partial K)$, there exists some $\boldsymbol \tau\in\boldsymbol{H}(\div \boldsymbol{\div },K; \mathbb{S})$ such that
\begin{equation}\label{eq:divdivinvtracee0}
\partial_t(\boldsymbol  t^{\intercal}\boldsymbol \tau\boldsymbol  n)+\boldsymbol  n^{\intercal}\boldsymbol \div\boldsymbol \tau=g, \quad
\|\boldsymbol{\tau}\|_{\boldsymbol{H}(\div \boldsymbol{\div })} \lesssim \|g\|_{H_e^{-3/2}(\partial K)}.
\end{equation}
The hidden constants depend only the shape of the domain $K$.
\end{lemma}

\subsection{Continuity across the boundary}
We then present a sufficient continuity condition for piecewise smoothing functions to be in $\boldsymbol{H}(\div \boldsymbol{\div },\Omega; \mathbb{S})$. Recall that $\mathcal {T}_h$ is a shape regular polygonal mesh of $\Omega$.

\begin{lemma}\label{lem:Hdivdivpatching}
Let $\boldsymbol \tau\in \boldsymbol  L^2(\Omega;\mathbb S)$ such that
\begin{enumerate}[(i)]
\item $\boldsymbol \tau|_K\in \boldsymbol{H}(\div \boldsymbol{\div },K; \mathbb{S})$ for each polygon $K\in\mathcal T_h$;

\smallskip
\item $(\boldsymbol  n^{\intercal}\boldsymbol \tau\boldsymbol  n)|_e\in L^2(e)$ is single-valued for each $e\in\mathcal E_h^i$;

\smallskip
\item $(\partial_{t_e}(\boldsymbol  t^{\intercal}\boldsymbol \tau\boldsymbol  n)+\boldsymbol  n_e^{\intercal}\boldsymbol \div\boldsymbol \tau)|_e\in L^2(e)$ is single-valued for each $e\in\mathcal E_h^i$;

\smallskip
\item $\boldsymbol \tau(\delta)$ is single-valued for each $\delta\in\mathcal V_h^i$,
\end{enumerate}
then $\boldsymbol \tau\in \boldsymbol{H}(\div \boldsymbol{\div },\Omega; \mathbb{S})$.
\end{lemma}
\begin{proof}
For any $v\in C_0^{\infty}(\Omega)$, it follows from the Green's identity \eqref{eq:greenidentitydivdiv} that
\begin{align*}
(\boldsymbol \tau, \nabla^2v)&=\sum_{K\in\mathcal T_h}(\div\boldsymbol \div\boldsymbol \tau, v)_K+\sum_{K\in\mathcal T_h}\sum_{e\in\mathcal E^i(K)}\sum_{\delta\in\partial e\cap\Omega}\sign_{e,\delta}(\boldsymbol  t^{\intercal}\boldsymbol \tau\boldsymbol  n)(\delta)v(\delta)\\
&\quad+\sum_{K\in\mathcal T_h}\sum_{e\in\mathcal E^i(K)}\left[(\boldsymbol  n^{\intercal}\boldsymbol \tau\boldsymbol  n, \partial_n v)_{e}-(\partial_{t}(\boldsymbol  t^{\intercal}\boldsymbol \tau\boldsymbol  n)+\boldsymbol  n^{\intercal}\boldsymbol \div\boldsymbol \tau,  v)_{e}\right].
\end{align*}
As each interior edge is repeated twice in the summation with opposite orientation and the trace of $\tau$ and vertex value $\tau$ is single valued, we get
\[
\langle\div\boldsymbol \div\boldsymbol \tau, v\rangle=\sum_{K\in\mathcal T_h}(\div\boldsymbol \div\boldsymbol \tau, v)_K,
\]
which ends the proof.
\end{proof}
Besides the continuity of the trace, we also impose the continuity of stress at vertices which is a sufficient but not necessary condition for functions in $\boldsymbol{H}(\div \boldsymbol{\div },\Omega; \mathbb{S})$. For example, by the complex \eqref{eq:divdivcomplexL2-intro} and Lemma \ref{lm:tauv}, for $\boldsymbol  \tau = \sym \boldsymbol  \curl \, \boldsymbol  v$ with $\bs v$ being a Lagrange element function, $\bs \tau\in \boldsymbol{H}(\div \boldsymbol{\div },\Omega; \mathbb{S})$ but is not continuous at vertices. Physically, $\boldsymbol  t^{\intercal}\boldsymbol \tau\boldsymbol  n$ represents the torsional moment which may have jump at vertices; see \cite[\S 3.4]{FengShi1996} and \cite[\S 3.4]{HuangShiXu2005}. {Sufficient and necessary conditions are presented in \cite[Proposition 3.6]{Fuhrer;Heuer;Niemi:2019ultraweak}.}

The continuity of stress at vertices is crucial for us to construct $\boldsymbol{H}(\div \boldsymbol{\div }, \Omega; \mathbb{S})$ conforming element in the classical triple \cite{Ciarlet1978}, which resembles the $\boldsymbol{H}(\boldsymbol{\div }, \Omega; \mathbb{S})$ conforming Hu-Zhang element for linear elasticity \cite{Hu;Zhang:2015family}.

\section{Conforming finite element spaces and complex}
In this section we construct conforming finite element spaces for $H(\div {\boldsymbol  \div},\Omega; \mathbb S)$ on triangles. We first present two polynomial complexes and reveal some decompositions of polynomial tensor and vector spaces. Then we construct the finite element space and prove the unisolvence. We further link standard finite element spaces to construct finite element div-div complex. Finally we extend the construction to the strain complex. 

\subsection{Polynomial complexes}
In this subsection, we shall consider polynomial spaces on a simply connected domain $D$. Without loss of generality, we assume $(0,0) \in D$.
\begin{lemma}
The polynomial complex
\begin{equation}\label{eq:divdivcomplexPoly}
\boldsymbol{RT}\autorightarrow{$\subset$}{} \mathbb P_{k+1}(D;\mathbb R^2)\autorightarrow{$\sym\boldsymbol\curl$}{} \mathbb P_k(D;\mathbb S) \autorightarrow{$\div\boldsymbol{\div}$}{} \mathbb P_{k-2}(D)\autorightarrow{}{}0
\end{equation}
is exact.
\end{lemma}
\begin{proof}
For any skew-symmetric $\boldsymbol \tau\in \mathcal C^2(D; \mathbb K)$, it can be written as $\boldsymbol  \tau =
\begin{pmatrix}
 0 & \phi\\
 -\phi & 0
\end{pmatrix}
$, then we have $\div\boldsymbol \div\boldsymbol \tau= \div\curl\phi = 0$. Hence
\[
\div\boldsymbol \div\sym\boldsymbol\curl \, \mathbb P_{k+1}(D;\mathbb R^2) = \div\boldsymbol \div\boldsymbol\curl \, \mathbb P_{k+1}(D;\mathbb R^2) = 0,
\]
\[
\div\boldsymbol{\div}\, \mathbb P_k(D;\mathbb S)=\div\boldsymbol{\div}\, \mathbb P_k(D;\mathbb M)=P_{k-2}(D).
\]
Furthermore by direct calculation
\[
\dim\mathbb P_k(D;\mathbb S)=\dim \sym\boldsymbol \curl \, \mathbb P_{k+1}(D;\mathbb R^2)+\dim \mathbb P_{k-2}(D),
\]
thus the complex \eqref{eq:divdivcomplexPoly} is exact.
\end{proof}

Define operator $\boldsymbol \pi_{RT}: \mathcal C^1(D; \mathbb R^2)\to \boldsymbol{RT}$ as
\[
\boldsymbol \pi_{RT}\boldsymbol  v:=\boldsymbol  v(0,0)+\frac{1}{2}(\div\boldsymbol  v)(0,0)\boldsymbol  x.
\]

The following complex is the generalization of the Koszul complex for vector functions. For linear elasticity, it can be constructed based on Poincar\'e operators found in~\cite{ChristiansenHuSande2020}. Here we give a straightforward proof. 

\begin{lemma}
The polynomial complex
\begin{equation}\label{eq:divdivKoszulcomplexPoly}
0\autorightarrow{$\subset$}{}\mathbb P_{k-2}(D) \autorightarrow{$\boldsymbol x\boldsymbol x^{\intercal}$}{} \mathbb P_k(D;\mathbb S) \autorightarrow{$\boldsymbol x^{\perp}$}{} \mathbb P_{k+1}(D;\mathbb R^2)\autorightarrow{$\boldsymbol \pi_{RT}$}{}\boldsymbol{RT}\autorightarrow{}{}0
\end{equation}
is exact.
\end{lemma}
\begin{proof}
Since $(\boldsymbol x\boldsymbol x^{\intercal})\boldsymbol x^{\perp}=\boldsymbol  0$ and $\boldsymbol \pi_{RT}(\boldsymbol \tau\boldsymbol x^{\perp})=\boldsymbol 0$ for any $\boldsymbol \tau\in\mathbb P_k(D;\mathbb S)$, thus \eqref{eq:divdivKoszulcomplexPoly} is a complex.
For any $\boldsymbol \tau\in\mathbb P_k(D;\mathbb S)$ satisfying $\boldsymbol \tau\boldsymbol x^{\perp}=\boldsymbol  0$, there exists $\boldsymbol  v\in \mathbb P_{k-1}(D;\mathbb R^2)$ such that $\boldsymbol \tau=\boldsymbol  v\boldsymbol  x^{\intercal}$. By the symmetry of $\boldsymbol \tau$,
\[
\boldsymbol  x(\boldsymbol  v^{\intercal}\boldsymbol x^{\perp})=(\boldsymbol  x\boldsymbol  v^{\intercal})\boldsymbol x^{\perp}=(\boldsymbol  v\boldsymbol  x^{\intercal})^{\intercal}\boldsymbol x^{\perp}=\boldsymbol  v\boldsymbol  x^{\intercal}\boldsymbol x^{\perp}=\boldsymbol 0,
\]
which indicates $\boldsymbol  v^{\intercal}\boldsymbol x^{\perp}=0$. Thus there exists $q\in\mathbb P_{k-2}(D)$ satisfying $\boldsymbol  v=q\boldsymbol  x$.
Hence $\boldsymbol \tau=q\boldsymbol  x\boldsymbol  x^{\intercal}$.

Next we show $\mathbb P_{k+1}(D;\mathbb R^2)\cap\ker(\boldsymbol \pi_{RT})=\mathbb P_k(D;\mathbb S)\boldsymbol x^{\perp}$.
For any $\boldsymbol  v\in\mathbb P_{k+1}(D;\mathbb R^2)\cap\ker(\boldsymbol \pi_{RT})$, since $\boldsymbol  v(0,0)=\boldsymbol 0$, there exist $\boldsymbol \tau_1\in\mathbb P_k(D;\mathbb S)$ and $q\in\mathbb P_k(D)$ such that
\[
\boldsymbol  v=\boldsymbol \tau_1\boldsymbol x^{\perp}+\begin{pmatrix}
0 & -q \\
q & 0
\end{pmatrix}\boldsymbol x^{\perp}=\boldsymbol \tau_1\boldsymbol x^{\perp}+q\boldsymbol  x.
\]
Noting that $\boldsymbol \pi_{RT}(\boldsymbol \tau_1\boldsymbol x^{\perp})=\boldsymbol 0$, we also have $\boldsymbol \pi_{RT}(q\boldsymbol  x)=\boldsymbol 0$. This means
\[
(\div(q\boldsymbol  x))(0,0)=0,\quad \textrm{ i.e. } \quad q(0,0)=0.
\]
Thus there exists $\boldsymbol  q_1\in\mathbb P_{k-1}(D;\mathbb R^2)$ such that $q=\boldsymbol  q_1^{\intercal}\boldsymbol x^{\perp}$.
Now take $\boldsymbol \tau=\boldsymbol \tau_1+2\sym(\boldsymbol  x\boldsymbol  q_1^{\intercal})\in\mathbb P_k(D;\mathbb S)$, then
\[
\boldsymbol \tau\boldsymbol x^{\perp}=\boldsymbol \tau_1\boldsymbol x^{\perp}+(\boldsymbol  x\boldsymbol  q_1^{\intercal}+\boldsymbol  q_1\boldsymbol  x^{\intercal})\boldsymbol x^{\perp}=\boldsymbol \tau_1\boldsymbol x^{\perp}+\boldsymbol  x q=\boldsymbol  v.
\]
Hence $\mathbb P_{k+1}(D;\mathbb R^2)\cap\ker(\boldsymbol \pi_{RT})=\mathbb P_k(D;\mathbb S)\boldsymbol x^{\perp}$ holds.

Apparently the operator $\boldsymbol \pi_{RT}: \mathbb P_{k+1}(D;\mathbb R^2)\to\boldsymbol {RT}$ is surjective as
$$
\boldsymbol \pi_{RT}\boldsymbol  v=\boldsymbol  v\quad \forall~\boldsymbol  v\in\boldsymbol{RT}.
$$
\end{proof}

Those two complexes \eqref{eq:divdivcomplexPoly} and \eqref{eq:divdivKoszulcomplexPoly} are connected as
\begin{equation}\label{eq:divdivcomplexPolydouble}
\xymatrix{
\boldsymbol{RT}\ar@<0.4ex>[r]^-{\subset} & \; \mathbb P_{k+1}(D;\mathbb R^2)\; \ar@<0.4ex>[r]^-{\sym\boldsymbol\curl}\ar@<0.4ex>[l]^-{\boldsymbol x}  & \; \mathbb P_k(D;\mathbb S) \ar@<0.4ex>[r]^-{\div\boldsymbol{\div}}\; \ar@<0.4ex>[l]^-{\boldsymbol x^{\bot}} & \; \mathbb P_{k-2}(D)  \; \ar@<0.4ex>[r]^-{} \ar@<0.4ex>[l]^-{\boldsymbol x\boldsymbol x^{\intercal}}
& 0 \ar@<0.4ex>[l]^-{\supset} }.
\end{equation}
Unlike the Koszul complex for vectors functions, we do not have the identity property applied to homogenous polynomials. Fortunately decomposition of polynomial spaces using Koszul and differential operators still holds.

First of all, we have the decomposition
\[
\mathbb P_{k+1}(D;\mathbb R^2) = \mathbb P_k(D;\mathbb S)\boldsymbol x^{\perp}\oplus\boldsymbol{RT}.
\]

%
Let
\[
\mathbb C_k(D; \mathbb S):=\sym\boldsymbol \curl \, \mathbb P_{k+1}(D;\mathbb R^2),\quad \mathbb C_k^{\oplus}(D; \mathbb S):=\boldsymbol  x\boldsymbol  x^{\intercal}\mathbb P_{k-2}(D).
\]
The dimensions are
\[
\dim\mathbb C_k(D; \mathbb S)=k^2+5k+3,\quad \dim\mathbb C_k^{\oplus}(D; \mathbb S)=\frac{1}{2}k(k-1).
\]
The following decomposition for the polynomial symmetric tensor is indispensable for our construction of div-div conforming finite elements.  
\begin{lemma}\label{lem:symmpolyspacedirectsum}
It holds
\[
\mathbb P_{k}(D;\mathbb S)=\mathbb C_k(D;\mathbb S)\oplus \mathbb C_k^{\oplus}(D;\mathbb S).
\]
And $\div\boldsymbol \div: \mathbb C_k^{\oplus}(D;\mathbb S)\to\mathbb P_{k-2}(D;\mathbb R^2)$ is a bijection.
\end{lemma}
\begin{proof}
Assume $q\in\mathbb P_{k-2}(D)$ satisfies $\boldsymbol x\boldsymbol  x^{\intercal}q\in\mathbb C_k(D;\mathbb S)$, which means
\[
\div\boldsymbol{\div}(\boldsymbol  x\boldsymbol  x^{\intercal}q)=0.
\]
Since
$
\boldsymbol{\div}(\boldsymbol  x\boldsymbol  x^{\intercal}q)=(\div(\boldsymbol  x q)+q)\boldsymbol  x
$,  we get
\[
\div(\boldsymbol  x q)+q=0.
\]
Then
\[
\div((x_1+x_2)\boldsymbol  x q)=(x_1+x_2)(\div(\boldsymbol  x q)+q)=0,
\]
which indicates $q=0$.
Hence $\mathbb C_k(D;\mathbb S)\cap\mathbb C_k^{\oplus}(D;\mathbb S)=\boldsymbol 0$.
Therefore we obtain the decomposition by the fact $\dim\mathbb P_{k}(D;\mathbb S)=\dim\mathbb C_k(D;\mathbb S)+\dim\mathbb C_k^{\oplus}(D;\mathbb S)$.

To prove the second result, we shall show a stronger result 
\begin{equation}\label{eq:divdivk}
\div\boldsymbol \div (\bs x\bs x^{\intercal} q) = (k+3)(k+2) q, \quad q\in \mathbb H_k(D).
\end{equation}
By Euler's formula for homogenous polynomial, we obtain $\div(\boldsymbol  x q)=(\boldsymbol x\cdot\nabla)q+2q = (k+2)q$. Then $
\boldsymbol{\div}(\boldsymbol  x\boldsymbol  x^{\intercal}q)=(\div(\boldsymbol  x q)+q)\boldsymbol  x = (k+3)q \boldsymbol x.
$
Computing $\div$ again and using $\div(\boldsymbol  x q)=(k+2)q$, we obtain \eqref{eq:divdivk}. 
\end{proof}

\begin{remark}\label{rm:Ek}
For a vector $\bs x = (x_1, x_2)$, introduce the rotation $\boldsymbol x^{\perp}= (x_2, -x_1)$. For the linear elasticity, we have the decomposition
\[
\mathbb P_{k}(D;\mathbb S)=\mathbb E_k(D;\mathbb S)\oplus \mathbb E_k^{\oplus}(D;\mathbb S),
\]
where, with $\bs \defm$ being the symmetric gradient operator,
\[
\mathbb E_k(D;\mathbb S):=\boldsymbol \defm\,\mathbb P_{k+1}(D;\mathbb R^2),\quad \mathbb E_k^{\oplus}(D;\mathbb S):=\boldsymbol  x^{\perp}(\boldsymbol  x^{\perp})^{\intercal}\mathbb P_{k-2}(D).
\]
\end{remark}

It is easy to see that
The polynomial complex
\begin{equation}\label{eq:hesscomplexPoly}
\mathbb P_{1}(D)\autorightarrow{$\subset$}{} \mathbb P_{k+1}(D)\autorightarrow{$\nabla^2$}{} \mathbb P_{k-1}(D;\mathbb S) \autorightarrow{$\boldsymbol{\rot}$}{} \mathbb P_{k-2}(D;\mathbb R^2)\autorightarrow{}{}0
\end{equation}
is exact, which the dual complex of the polynomial complex \eqref{eq:divdivcomplexPoly}.

Define operator $\pi_{1}: \mathcal C^1(D)\to \mathbb P_{1}(D)$ as
\[
\pi_{1}v:=v(0,0)+\boldsymbol  x^{\intercal}(\nabla v)(0,0).
\]

\begin{lemma}
The polynomial complex
\begin{equation}\label{eq:hessKoszulcomplexPoly}
\resizebox{.9\hsize}{!}{$
\boldsymbol{0}\autorightarrow{$\subset$}{} \mathbb P_{k-2}(D;\mathbb R^2)\autorightarrow{$\sym(\boldsymbol x^{\perp}\otimes\bs v)$}{} \mathbb P_{k-1}(D;\mathbb S) \autorightarrow{$\boldsymbol x^{\intercal}\boldsymbol \tau\boldsymbol x$}{} \mathbb P_{k+1}(D)\autorightarrow{$\pi_1$}{}\mathbb P_{1}(D)\autorightarrow{}{}0
$}
\end{equation}
is exact.
\end{lemma}
\begin{proof}
It is easy to check that $\boldsymbol x^{\intercal}\mathbb P_{k-1}(D;\mathbb S)\boldsymbol x=\mathbb P_{k+1}(D)\cap\ker(\pi_1)$, and
\[
\dim\boldsymbol x^{\intercal}\mathbb P_{k-1}(D;\mathbb S)\boldsymbol x=\dim\mathbb P_{k+1}(D)-3=\frac{1}{2}(k^2+5k).
\]
Noting that
\[
\dim\sym(\boldsymbol x^{\perp}\otimes\mathbb P_{k-2}(D;\mathbb R^2))=\dim\mathbb P_{k-2}(D;\mathbb R^2)=k^2-k,
\]
we get
\[
\dim\sym(\boldsymbol x^{\perp}\otimes\mathbb P_{k-2}(D;\mathbb R^2))+\dim\boldsymbol x^{\intercal}\mathbb P_{k-1}(D;\mathbb S)\boldsymbol x=\dim\mathbb P_{k-1}(D;\mathbb S).
\]
 Hence the complex \eqref{eq:hessKoszulcomplexPoly} is exact. 
\end{proof}

\begin{lemma}\label{lem:rot}
It holds
\begin{equation}\label{eq:polydecompsymx}
\mathbb P_{k-1}(D;\mathbb S)=\nabla^2\mathbb P_{k+1}(D)\oplus \sym(\boldsymbol x^{\perp}\otimes\mathbb P_{k-2}(D;\mathbb R^2)).
\end{equation}
And $\bs\rot: \sym(\boldsymbol x^{\perp}\otimes\mathbb P_{k-2}(D;\mathbb R^2)) \to  \mathbb P_{k-2}(D;\mathbb R^2)$ is a bijection.
\end{lemma}
\begin{proof}
Due to the fact
\[
\dim\mathbb P_{k-1}(D;\mathbb S)=\dim\nabla^2\mathbb P_{k+1}(D) + \dim\sym(\boldsymbol x^{\perp}\otimes\mathbb P_{k-2}(D;\mathbb R^2)),
\] 
it is sufficient to prove $\nabla^2\mathbb P_{k+1}(D)\cap \sym(\boldsymbol x^{\perp}\otimes\mathbb P_{k-2}(D;\mathbb R^2))=\bs0$ for \eqref{eq:polydecompsymx}. For any $q\in\mathbb P_{k+1}(D)$ satisfying $\nabla^2q\in\sym(\boldsymbol x^{\perp}\otimes\mathbb P_{k-2}(D;\mathbb R^2))$, it holds
\[
(\bs x\cdot\nabla)^2q=\boldsymbol x^{\intercal}(\nabla^2q)\boldsymbol x=0.
\]
Hence $q\in \mathbb P_{1}(D)$, which yields the decomposition \eqref{eq:polydecompsymx}. 

On the other hand, it follows from the direct sum of \eqref{eq:polydecompsymx} that $\bs\rot: \sym(\boldsymbol x^{\perp}\otimes\mathbb P_{k-2}(D;\mathbb R^2)) \to  \mathbb P_{k-2}(D;\mathbb R^2)$ is injective.
Furthermore, we get from the complex~\eqref{eq:hesscomplexPoly} that
\[
\bs\rot\sym(\boldsymbol x^{\perp}\otimes\mathbb P_{k-2}(D;\mathbb R^2)) =\bs\rot\mathbb P_{k-1}(D;\mathbb S)=\mathbb P_{k-2}(D;\mathbb R^2).
\]
This ends the proof.
\end{proof}

\subsection{Finite element spaces for symmetric tensors}
Let $K$ be a triangle, and $b_K$ be the cubic bubble function, i.e., $b_K\in \mathbb P_3(K)\cap H_0^1(K)$.
Take the space of shape functions
\[
\boldsymbol \Sigma_{\ell,k}(K):= \mathbb C_{\ell}(K;\mathbb S)\oplus\mathbb C_k^{\oplus}(K;\mathbb S)
\]
with $k\geq 3$ and $\ell\geq k-1$.
 By Lemma \ref{lem:symmpolyspacedirectsum}, we have
\[
\mathbb P_{\min\{\ell,k\}}(K;\mathbb S)\subseteq\boldsymbol \Sigma_{\ell,k}(K) \subseteq \mathbb P_{\max\{\ell,k\}}(K;\mathbb S) \quad\textrm{ and }\quad \boldsymbol \Sigma_{k,k}(K)=\mathbb P_k(K;\mathbb S).
\]
The most interesting cases are $\ell=k-1$ and $\ell = k$ which correspond to RT and BDM $H(\div)$-conforming elements for the vector functions, respectively. 

The degrees of freedom are given by
\begin{align}
\boldsymbol \tau (\delta) & \quad\forall~\delta\in \mathcal V(K), \label{Hdivdivfemdof1}\\
(\boldsymbol  n^{\intercal}\boldsymbol \tau\boldsymbol  n, q)_e & \quad\forall~q\in\mathbb P_{\ell-2}(e),  e\in\mathcal E(K),\label{Hdivdivfemdof2}\\
(\partial_{t}(\boldsymbol  t^{\intercal}\boldsymbol \tau\boldsymbol  n)+\boldsymbol  n^{\intercal}\boldsymbol \div\boldsymbol \tau, q)_e & \quad\forall~q\in\mathbb P_{\ell-1}(e),  e\in\mathcal E(K),\label{Hdivdivfemdof3}\\
(\boldsymbol \tau, \boldsymbol \varsigma)_K & \quad\forall~\boldsymbol \varsigma\in\nabla^2\mathbb P_{k-2}(K)\oplus \sym (\bs x^{\perp}\otimes \mathbb P_{\ell-2}(K;\mathbb R^2)). \label{Hdivdivfemdof4}
\end{align}

Before we prove the unisolvence, we give some characterization of space of shape functions.
\begin{lemma}\label{lm:trace}
For any $\boldsymbol \tau\in\boldsymbol \Sigma_{\ell,k}(K)$, we have 
\begin{enumerate}
 \item $\boldsymbol \tau\boldsymbol x^{\perp} \in \mathbb P_{\ell+1}(K;\mathbb R^2)$ 
 
 \smallskip
 \item $\boldsymbol  n^{\intercal}\boldsymbol \tau\boldsymbol  n|_e\in\mathbb P_{\ell}(e)\quad\forall~e\in\mathcal E(K)$

\smallskip
 \item $(\partial_{t}(\boldsymbol  t^{\intercal}\boldsymbol \tau\boldsymbol  n)+\boldsymbol  n^{\intercal}\boldsymbol \div\boldsymbol \tau)|_e\in\mathbb P_{\ell-1}(e)\quad\forall~e\in\mathcal E(K)$.
\end{enumerate}
%
%
\end{lemma}
\begin{proof}
(1) is a direct consequence of the Koszul complex \eqref{eq:divdivKoszulcomplexPoly}. 
Take any $\boldsymbol \tau=\boldsymbol  x\boldsymbol  x^{\intercal}q\in\mathbb C_k^{\oplus}(K;\mathbb S)$ with $q\in\mathbb P_{k-2}(K)$.
Since $\boldsymbol  n^{\intercal}\boldsymbol  x$ is constant on each edge of $K$,
\[
 \boldsymbol  n^{\intercal}\boldsymbol \tau\boldsymbol  n|_e=(\boldsymbol  n^{\intercal}\boldsymbol  x)^2q\in\mathbb P_{k-2}(e),
\]
\[
 (\partial_{t}(\boldsymbol  t^{\intercal}\boldsymbol \tau\boldsymbol  n)+\boldsymbol  n^{\intercal}\boldsymbol \div\boldsymbol \tau)|_e=\boldsymbol  n^{\intercal}\boldsymbol  x(\partial_{t}\left(\boldsymbol  t^{\intercal}\boldsymbol x q)+\div(\boldsymbol  x q)+q\right)\in\mathbb P_{k-2}(e).
\]
Thus the results (2) and (3) hold from the requirement $\ell\geq k-1$.
\end{proof}

We now prove the unisolvence as follows.

\begin{lemma}\label{lem:unisovlenHdivdivfem}
The degrees of freedom \eqref{Hdivdivfemdof1}-\eqref{Hdivdivfemdof4} are unisolvent for $\boldsymbol \Sigma_{\ell,k}(K)$.
\end{lemma}
\begin{proof}
We first count the number of the degrees of freedom \eqref{Hdivdivfemdof1}-\eqref{Hdivdivfemdof4}  and the dimension of the space, i.e., $\dim\boldsymbol \Sigma_{\ell,k}(K)$.
Both of them are $$\ell^2+5\ell+3+\frac{1}{2}k(k-1).$$


Then suppose all the degrees of freedom \eqref{Hdivdivfemdof1}-\eqref{Hdivdivfemdof4} applied to $\boldsymbol \tau$ vanish. We are going to prove the function $\boldsymbol \tau = 0$. 

\medskip

\noindent {\em Step 1. Trace is vanished.} By the vanishing degrees of freedom \eqref{Hdivdivfemdof1}-\eqref{Hdivdivfemdof3} and (2)-(3) in Lemma \ref{lm:trace}, we get $(\boldsymbol n^{\intercal}\boldsymbol \tau\boldsymbol n)|_{\partial K}=0$ and $(\partial_{t}(\boldsymbol  t^{\intercal}\boldsymbol \tau\boldsymbol  n)+\boldsymbol  n^{\intercal}\boldsymbol \div\boldsymbol \tau)|_{\partial K}=0$.

\medskip

\noindent {\em Step 2. Divdiv is vanished.}
For any $q\in\mathbb P_{k-2}(K)$, it holds from the Green's identity \eqref{eq:greenidentitydivdiv} that
\begin{align*}
(\div\boldsymbol \div\boldsymbol \tau, q)_K&=(\boldsymbol \tau, \nabla^2q)_K-\sum_{e\in\mathcal E(K)}\sum_{\delta\in\partial e}\sign_{e,\delta}(\boldsymbol  t^{\intercal}\boldsymbol \tau\boldsymbol  n)(\delta)q(\delta) \\
&\quad-\sum_{e\in\mathcal E(K)}\left[(\boldsymbol  n^{\intercal}\boldsymbol \tau\boldsymbol  n, \partial_n q)_{e}-(\partial_{t}(\boldsymbol  t^{\intercal}\boldsymbol \tau\boldsymbol  n)+\boldsymbol  n^{\intercal}\boldsymbol \div\boldsymbol \tau,  q)_{e}\right].
\end{align*}
Since the trace is zero, $\bs \tau$ is zero at vertices, and $(\bs \tau, \nabla^2 q)_K = 0$ from \eqref{Hdivdivfemdof4},  we conclude $\div\boldsymbol \div\boldsymbol \tau=0$. 

\medskip

\noindent {\em Step 3. Kernel of divdiv is vanished.}
Thus by the polynomial complex \eqref{eq:divdivcomplexPoly}, there exists $\boldsymbol  v\in\mathbb P_{\ell+1}(K;\mathbb R^2)/\boldsymbol {RT}$ such that
\[
\boldsymbol \tau=\sym\boldsymbol \curl \, \boldsymbol  v\quad \textrm{ and }\quad Q_{0}^e(\boldsymbol  n^{\intercal}\boldsymbol  v)=0\quad\forall~e\in\mathcal E(K).
\]
Here we can take $Q_{0}^e(\boldsymbol  n^{\intercal}\boldsymbol  v)=0$ thanks to the degree of freedom of the lowest order Raviart-Thomas element \cite{RaviartThomas1977}. We will prove $\bs v = 0$ by similar procedure. 

By Lemma \ref{lm:tauv}, the fact $(\boldsymbol  n^{\intercal}\boldsymbol \tau\boldsymbol  n)|_{\partial K}=0$ implies
\[
\partial_t(\boldsymbol  n^{\intercal}\boldsymbol  v)|_{\partial K}= (\boldsymbol  n^{\intercal}\boldsymbol \tau\boldsymbol  n)|_{\partial K}=0.
\]
Hence $\boldsymbol  n^{\intercal}\boldsymbol  v|_{\partial K}=0$.
This also means $\boldsymbol  v(\delta)=\boldsymbol 0$ for each $\delta\in\mathcal V(K)$.

Again by Lemma \ref{lm:tauv}, since
\begin{equation*}
\partial_{t}(\boldsymbol  t^{\intercal}\boldsymbol \tau\boldsymbol  n)+\boldsymbol  n^{\intercal}\boldsymbol \div\boldsymbol \tau= \partial_t(\boldsymbol  t^{\intercal}\partial_t\boldsymbol  v)
\end{equation*}
and $(\partial_{t}(\boldsymbol  t^{\intercal}\boldsymbol \tau\boldsymbol  n)+\boldsymbol  n^{\intercal}\boldsymbol \div\boldsymbol \tau)|_{\partial K}=0$, we acquire
\[
\partial_{tt}(\boldsymbol  t^{\intercal}\boldsymbol  v)|_{\partial K}=0.
\]
That is $\boldsymbol  t^{\intercal}\boldsymbol  v|_e\in\mathbb P_1(e)$ on each edge $e\in\mathcal E(K)$. Noting that $\boldsymbol  v(\delta)=\boldsymbol 0$ for each $\delta\in\mathcal V(K)$, we get $\boldsymbol  t^{\intercal}\boldsymbol  v|_{\partial K}=0$ and consequently $\boldsymbol  v|_{\partial K}=\boldsymbol 0$, i.e., 
$$
\boldsymbol  v = b_K\psi_{\ell-2},\quad \text{for some } \psi_{\ell-2} \in \mathbb P_{\ell-2}(K;\mathbb R^2).
$$

We then use the fact $\rot: \sym(\boldsymbol x^{\perp}\otimes\mathbb P_{\ell-2}(K;\mathbb R^2)) \to  \mathbb P_{\ell-2}(K;\mathbb R^2)$ is bijection, cf. Lemma \ref{lem:rot}, to find $\phi_{\ell-2}$ s.t. $\rot (\sym\boldsymbol x^{\perp}\otimes\phi_{\ell -2}) = \psi_{\ell-2}$.

Finally we finish the proof by choosing $\boldsymbol \varsigma = \sym (\boldsymbol x^{\perp}\otimes \phi_{\ell -2})$ in \eqref{Hdivdivfemdof4}. The fact 
$$
(\bs \tau, \boldsymbol \varsigma)_K = (\sym\boldsymbol \curl b_K\psi_{\ell-2},  \sym (\boldsymbol x^{\perp}\otimes \phi_{\ell -2}))_K = ( b_K\psi_{\ell-2}, \psi_{\ell-2} )_K = 0
$$
will imply $\psi_{\ell-2} = 0$ and consequently $\bs v = 0, \bs \tau = 0$.  
\end{proof}

%
%
%

\subsection{Finite element $\div$-$\div$ complex}
Recall the $\div$-$\div$ Hilbert complexes with different regularity
\begin{equation}\label{eq:divdivcomplexL2}
\resizebox{.9\hsize}{!}{$
\boldsymbol  {RT}\autorightarrow{$\subset$}{} \boldsymbol  H^1(K;\mathbb R^2)\autorightarrow{$\sym\boldsymbol \curl$}{} \boldsymbol{H}(\div\boldsymbol{\div},K; \mathbb{S}) \autorightarrow{$\div \boldsymbol{\div }$}{} L^2(K)\autorightarrow{}{}0$},
\end{equation}
\begin{equation}\label{eq:divdivcomplexH2}
\boldsymbol  {RT}\autorightarrow{$\subset$}{} \boldsymbol  H^3(K;\mathbb R^2)\autorightarrow{$\sym\boldsymbol \curl$}{} \boldsymbol{H}^2(K; \mathbb{S}) \autorightarrow{$\div\boldsymbol{\div}$}{} L^2(K)\autorightarrow{}{}0.
\end{equation}
Both complexes \eqref{eq:divdivcomplexL2} and \eqref{eq:divdivcomplexH2} are exact; see \cite{ChenHuang2018}.

We have constructed finite element spaces for $\boldsymbol{H}(\div\boldsymbol{\div},K; \mathbb{S}) $. Now we define a vectorial $H^1$-conforming finite element. Let $\boldsymbol V_{\ell+1}(K):=\mathbb P_{\ell+1}(K;\mathbb R^2)$ with $\ell\geq2$.
The local degrees of freedom are given by
\begin{align}
\boldsymbol  v (\delta), \nabla\boldsymbol  v (\delta) & \quad\forall~\delta\in \mathcal V(K), \label{HermitfemVdof1}\\
(\boldsymbol  v, \boldsymbol  q)_e & \quad\forall~\boldsymbol  q\in\mathbb P_{\ell-3}(e;\mathbb R^2),  e\in\mathcal E(K),\label{HermitfemVdof2}\\
(\boldsymbol  v, \boldsymbol  q)_K & \quad\forall~\boldsymbol  q\in\mathbb P_{\ell-2}(K;\mathbb R^2). \label{HermitfemVdof3}
\end{align}
This finite element is just the vectorial Hermite element \cite{BrennerScott2008,Ciarlet1978}. 

\begin{lemma}
For any triangle $K$, both the polynomial complexes
\begin{equation}\label{eq:divdivcomplexPolyvar}
\boldsymbol  {RT}\autorightarrow{$\subset$}{} \boldsymbol V_{\ell+1}(K)\autorightarrow{$\sym\boldsymbol \curl$}{} \boldsymbol \Sigma_{\ell,k}(K) \autorightarrow{$\div\boldsymbol{\div}$}{} \mathbb P_{k-2}(K)\autorightarrow{}{}0
\end{equation}
and
\begin{equation}\label{eq:divdivcomplexPolyvar0}
\boldsymbol 0\autorightarrow{$\subset$}{} \mathring{\boldsymbol V}_{\ell+1}(K)\autorightarrow{$\sym\boldsymbol \curl$}{} \mathring{\boldsymbol \Sigma}_{\ell,k}(K) \autorightarrow{$\div\boldsymbol{\div}$}{} \mathring{\mathbb P}_{k-2}(K)\autorightarrow{}{}0
\end{equation}
are exact, where
\begin{align*}
\mathring{\boldsymbol V}_{\ell+1}(K)&:=\{\boldsymbol  v\in\boldsymbol V_{\ell+1}(K): \textrm{all degrees of freedom } \eqref{HermitfemVdof1}-\eqref{HermitfemVdof2} \textrm{ vanish}\},
\\
\mathring{\boldsymbol \Sigma}_{\ell,k}(K)&:=\{\boldsymbol  \tau\in\boldsymbol \Sigma_{\ell,k}(K): \textrm{all degrees of freedom } \eqref{Hdivdivfemdof1}-\eqref{Hdivdivfemdof3} \textrm{ vanish}\},
\\
\mathring{\mathbb P}_{k-2}(K)&:=\mathbb P_{k-2}(K)/\mathbb P_{1}(K).
\end{align*}
\end{lemma}
\begin{proof}
The exactness of the complex \eqref{eq:divdivcomplexPolyvar} follows from the exactness of the complex~\eqref{eq:divdivcomplexPoly} and Lemma~\ref{lem:symmpolyspacedirectsum}.

By the proof of Lemma~\ref{lem:unisovlenHdivdivfem}, we have $\mathring{\boldsymbol \Sigma}_{\ell,k}(K)\cap\ker(\div\boldsymbol{\div})=\sym\boldsymbol \curl\mathring{\boldsymbol V}_{\ell+1}(K)$. This also means
\[
\dim\div\boldsymbol{\div}\mathring{\boldsymbol \Sigma}_{\ell,k}(K)=\dim\mathring{\boldsymbol \Sigma}_{\ell,k}(K)-\dim\mathring{\boldsymbol V}_{\ell+1}(K)=\frac{1}{2}k(k-1)-3=\dim\mathring{\mathbb P}_{k-2}(K).
\]
Due to the Green's identity \eqref{eq:greenidentitydivdiv}, we get $\div\boldsymbol{\div}\mathring{\boldsymbol \Sigma}_{\ell,k}(K)\subseteq \mathring{\mathbb P}_{k-2}(K)$, which ends the proof.
\end{proof}

To show the commutative diagram for the polynomial complex \eqref{eq:divdivcomplexPolyvar}, we introduce $\boldsymbol \Pi_K: \boldsymbol{H}^2(K; \mathbb{S})\to\boldsymbol \Sigma_{\ell,k}(K)$ be the nodal interpolation operator based on the degrees of freedom \eqref{Hdivdivfemdof1}-\eqref{Hdivdivfemdof4}. We have $\boldsymbol \Pi_K\boldsymbol \tau=\boldsymbol \tau$ for any $\boldsymbol \tau\in\mathbb P_{\min\{\ell,k\}}(K;\mathbb S)$, and
\begin{equation}\label{eq:PiKestimate}
\|\boldsymbol \tau-\boldsymbol \Pi_K\boldsymbol \tau\|_{0,K}+h_K|\boldsymbol \tau-\boldsymbol \Pi_K\boldsymbol \tau|_{1,K}+h_K^2|\boldsymbol \tau-\boldsymbol \Pi_K\boldsymbol \tau|_{2,K}\lesssim h_K^{s}|\boldsymbol \tau|_{s, K}
\end{equation}
for any $\boldsymbol \tau\in\boldsymbol{H}^s(K; \mathbb{S})$
with $2\leq s\leq \min\{\ell,k\}+1$. It follows from the Green's identity \eqref{eq:greenidentitydivdiv} that
\begin{equation}\label{eq:PiKcd}
\div\boldsymbol \div(\boldsymbol \Pi_K\boldsymbol \tau)=Q_{k-2}^K\div\boldsymbol \div\boldsymbol \tau\quad\forall~\boldsymbol \tau\in\boldsymbol{H}^2(K; \mathbb{S}).
\end{equation}
Let $\widetilde{\boldsymbol  I}_K: \boldsymbol{H}^3(K; \mathbb{R}^2)\to\boldsymbol V_{\ell+1}(K)$
be the nodal interpolation operator based on the degrees of freedom \eqref{HermitfemVdof1}-\eqref{HermitfemVdof3}.
We have $\widetilde{\boldsymbol  I}_K\boldsymbol  q=\boldsymbol  q$ for any $\boldsymbol  q\in\mathbb P_{\ell+1}(K;\mathbb R^2)$, and
\begin{equation}\label{eq:IKtildeestimate}
\|\boldsymbol  v-\widetilde{\boldsymbol  I}_K\boldsymbol  v\|_{0,K}+h_K|\boldsymbol  v-\widetilde{\boldsymbol  I}_K\boldsymbol  v|_{1,K} \lesssim h_K^{s}|\boldsymbol  v|_{s, K} \quad\forall~\boldsymbol  v\in \boldsymbol{H}^s(K; \mathbb{R}^2)
\end{equation}
with $3\leq s\leq \ell+2$.
Then we define $\boldsymbol  I_K: \boldsymbol{H}^3(K; \mathbb{R}^2)\to\boldsymbol V_{\ell+1}(K)$ by modifying $\widetilde{\boldsymbol  I}_K$.
By \eqref{eq:trace1} and \eqref{eq:trace2},
clearly we have $\boldsymbol \Pi_K(\sym\boldsymbol \curl \, \boldsymbol  v) - \sym\boldsymbol \curl(\widetilde{\boldsymbol  I}_K\boldsymbol  v)\in\mathring{\boldsymbol \Sigma}_{\ell,k}(K)$ for any $\boldsymbol  v\in\boldsymbol{H}^3(K; \mathbb{R}^2)$. And it holds from \eqref{eq:PiKcd} that
\[
\div\boldsymbol \div(\boldsymbol \Pi_K(\sym\boldsymbol \curl \, \boldsymbol  v) - \sym\boldsymbol \curl(\widetilde{\boldsymbol  I}_K\boldsymbol  v))=0.
\]
Thus using the complex \eqref{eq:divdivcomplexPolyvar0}, there exists $\widetilde{\boldsymbol  v}\in\mathring{\boldsymbol V}_{\ell+1}(K)$ satisfying
\[
\sym\boldsymbol\curl\widetilde{\boldsymbol  v}=\boldsymbol \Pi_K(\sym\boldsymbol \curl \, \boldsymbol  v) - \sym\boldsymbol \curl(\widetilde{\boldsymbol  I}_K\boldsymbol  v),
\]
\[
 \|\widetilde{\boldsymbol v}\|_{0,K}\lesssim h_K\|\boldsymbol \Pi_K(\sym\boldsymbol \curl \, \boldsymbol  v) - \sym\boldsymbol \curl(\widetilde{\boldsymbol  I}_K\boldsymbol  v)\|_{0,K}.
\]
Let $\boldsymbol  I_K\boldsymbol  v:=\widetilde{\boldsymbol  I}_K\boldsymbol  v+\widetilde{\boldsymbol  v}$. Apparently $\boldsymbol  I_K\boldsymbol  q=\boldsymbol  q$ for any $\boldsymbol  q\in\mathbb P_{\ell+1}(K;\mathbb R^2)$, and
\begin{equation}\label{eq:IKcd}
\sym\boldsymbol\curl(\boldsymbol  I_K\boldsymbol  v)=\boldsymbol \Pi_K(\sym\boldsymbol \curl \, \boldsymbol  v) \quad\forall~\boldsymbol  v\in\boldsymbol{H}^3(K; \mathbb{R}^2).
\end{equation}
It follows from  \eqref{eq:IKtildeestimate} and \eqref{eq:PiKestimate} that
\begin{equation}\label{eq:IKestimate}
\|\boldsymbol  v-\boldsymbol  I_K\boldsymbol  v\|_{0,K}+h_K|\boldsymbol  v-\boldsymbol  I_K\boldsymbol  v|_{1,K} \lesssim h_K^{s}|\boldsymbol  v|_{s, K} \quad\forall~\boldsymbol  v\in \boldsymbol{H}^s(K; \mathbb{R}^2).
\end{equation}
with $3\leq s\leq \ell+2$.

In summary, we have the following commutative diagram for the local finite element complex \eqref{eq:divdivcomplexPolyvar}
$$
\begin{array}{c}
\xymatrix{
  \boldsymbol{RT} \ar[r]^-{\subset} & \boldsymbol  H^3(K;\mathbb R^2) \ar[d]^{\boldsymbol{I}_K} \ar[r]^-{\sym\boldsymbol \curl}
                & \boldsymbol{H}^2(K; \mathbb{S}) \ar[d]^{\boldsymbol{\Pi}_K}   \ar[r]^-{\div\boldsymbol{\div}} & \ar[d]^{Q_K}L^2(K) \ar[r]^{} & 0 \\
 \boldsymbol{RT} \ar[r]^-{\subset} & \boldsymbol V_{\ell+1}(K) \ar[r]^{\sym\boldsymbol \curl}
                &  \boldsymbol \Sigma_{\ell,k}(K)   \ar[r]^{\div\boldsymbol{\div}} &  \mathbb P_{k-2}(K) \ar[r]^{}& 0    }
\end{array}
$$
with $Q_K:=Q_{k-2}^K$.

We then glue local finite element spaces to get global conforming spaces.
Define
\begin{align*}
\boldsymbol  V_h:=\{\boldsymbol  v_h\in \boldsymbol  H^1(\Omega;\mathbb R^2):&\, \boldsymbol  v_h|_K\in \mathbb P_{\ell+1}(K;\mathbb R^2) \textrm{ for each } K\in\mathcal T_h, \\
& \textrm{ all the degrees of freedom } \eqref{HermitfemVdof1}-\eqref{HermitfemVdof2}  \textrm{ are single-valued}\},
\end{align*}
\begin{align*}
\boldsymbol\Sigma_h:=\{\boldsymbol \tau_h\in \boldsymbol  L^2(\Omega;\mathbb S):&\, \boldsymbol \tau_h|_K\in \boldsymbol \Sigma_{\ell,k}(K) \textrm{ for each } K\in\mathcal T_h, \\
& \textrm{ all the degrees of freedom } \eqref{Hdivdivfemdof1}-\eqref{Hdivdivfemdof3} \textrm{ are single-valued}\},
\end{align*}
\[
\mathcal Q_h :=\mathbb P_{k-2}(\mathcal T_h)=\{q_h\in L^2(\Omega): q_h|_K\in \mathbb P_{k-2}(K) \textrm{ for each } K\in\mathcal T_h\}.
\]
Due to Lemma~\ref{lem:Hdivdivpatching}, the finite element space $\boldsymbol \Sigma_h\subset\boldsymbol{H}(\div \boldsymbol{\div },\Omega; \mathbb{S})$.
Let $\boldsymbol  I_h: \boldsymbol{H}^3(\Omega; \mathbb{S})\to\boldsymbol V_h$, $\boldsymbol \Pi_h: \boldsymbol{H}^2(\Omega; \mathbb{S})\to\boldsymbol\Sigma_h$ and $Q_h^{m}: L^2(\Omega)\to\mathbb P_{m}(\mathcal T_h)$ be defined by $(\boldsymbol  I_h\boldsymbol  v)|_K:=\boldsymbol  I_K(\boldsymbol  v|_K)$, $(\boldsymbol \Pi_h\boldsymbol \tau)|_K:=\boldsymbol \Pi_K(\boldsymbol \tau|_K)$  and $(Q_h^{m} q)|_K:=Q_{m}^K(q|_K)$ for each $K\in\mathcal T_h$, respectively. When the degree is clear from the context, we will simply write the $L^2$-projection $Q_h^{k-2}$ as $Q_h$.

As direct results of \eqref{eq:PiKcd} and \eqref{eq:IKcd}, we have
\begin{equation}\label{eq:Pihcd}
\div\boldsymbol \div(\boldsymbol \Pi_h\boldsymbol \tau)=Q_{h}\div\boldsymbol \div\boldsymbol \tau\quad\forall~\boldsymbol \tau\in\boldsymbol{H}^2(\Omega; \mathbb{S}),
\end{equation}
\begin{equation}\label{eq:Ihcd}
\sym\boldsymbol\curl(\boldsymbol  I_h\boldsymbol  v)=\boldsymbol \Pi_h(\sym\boldsymbol \curl \, \boldsymbol  v) \quad\forall~\boldsymbol  v\in\boldsymbol{H}^3(\Omega; \mathbb{R}^2).
\end{equation}

\begin{lemma}
The finite element complex
\begin{equation}\label{eq:divdivcomplexfem}
\boldsymbol  {RT}\autorightarrow{$\subset$}{} \boldsymbol  V_h\autorightarrow{$\sym\boldsymbol \curl$}{} \boldsymbol \Sigma_h \autorightarrow{$\div\boldsymbol{\div}$}{} \mathcal Q_h\autorightarrow{}{}0
\end{equation}
is exact. Moreover, we have the commutative diagram
\begin{equation}\label{eq:divdivcdfem}
\begin{array}{c}
\xymatrix{
  \boldsymbol{RT} \ar[r]^-{\subset} & \boldsymbol  H^3(\Omega;\mathbb R^2) \ar[d]^{\boldsymbol{I}_h} \ar[r]^-{\sym\boldsymbol \curl}
                & \boldsymbol{H}^2(\Omega; \mathbb{S}) \ar[d]^{\boldsymbol{\Pi}_h}   \ar[r]^-{\div\boldsymbol{\div}} & \ar[d]^{Q_{h}}L^2(\Omega) \ar[r]^{} & 0 \\
 \boldsymbol{RT} \ar[r]^-{\subset} & \boldsymbol V_{h} \ar[r]^{\sym\boldsymbol \curl}
                &  \boldsymbol \Sigma_{h}   \ar[r]^{\div\boldsymbol{\div}} &  \mathcal Q_h \ar[r]^{}& 0    }
\end{array}.
\end{equation}
\end{lemma}
\begin{proof}
By the complex \eqref{eq:divdivcomplexH2}, for any $q_h\in \mathcal Q_h$, there exists $\boldsymbol \tau \in\boldsymbol{H}^2(\Omega; \mathbb{S})$ satisfying $\div\boldsymbol \div\boldsymbol \tau=q_h$. Then it follows from \eqref{eq:Pihcd} that
\[
\div\boldsymbol \div(\boldsymbol  \Pi_h\boldsymbol \tau)=Q_h\div\boldsymbol \div\boldsymbol \tau=q_h.
\]
Hence $\div\boldsymbol{\div}\boldsymbol \Sigma_h = \mathcal Q_h$.
On the other hand, by counting we get
\[
\dim\boldsymbol \Sigma_h=3\#\mathcal V_h+(2\ell-1)\#\mathcal E_h+\ell(\ell-1)\#\mathcal T_h+\frac{1}{2}(k+2)(k-3)\#\mathcal T_h,
\]
\[
\dim\sym\boldsymbol \curl \, \boldsymbol  V_h=6\#\mathcal V_h+(2\ell-4)\#\mathcal E_h+\ell(\ell-1)\#\mathcal T_h-3,
\]
\[
\dim\div\boldsymbol{\div}\boldsymbol \Sigma_h=\dim\mathbb P_{k-2}(\mathcal T_h)=\frac{1}{2}k(k-1)\#\mathcal T_h.
\]
Here $\#\mathcal S$ means the number of the elements in the finite set $\mathcal S$.
It follows from the Euler's formula $\#\mathcal E_h+1=\#\mathcal V_h+\#\mathcal T_h$ that
\[
\dim\boldsymbol \Sigma_h=\dim\sym\boldsymbol \curl \, \boldsymbol  V_h+\dim\div\boldsymbol{\div}\boldsymbol \Sigma_h.
\]
Therefore the complex \eqref{eq:divdivcomplexfem} is exact.

The commutative diagram \eqref{eq:divdivcdfem} follows from \eqref{eq:Pihcd} and \eqref{eq:Ihcd}.
\end{proof}

\begin{remark}\rm
Using the smoothing procedure \cite{ArnoldFalkWinther2006}, the natural interpolations in the commutative diagram \eqref{eq:divdivcdfem} can be refined to quasi-interpolations and the top one can be replaced by the complex \eqref{eq:divdivcomplexL2} with minimal regularity.
\end{remark}

\subsection{Conforming finite element spaces for strain complex}
In the application of linear elasticity, the strain complex is more relevant. 
As the rotated version of \eqref{eq:greenidentitydivdiv}, we get the Green's identity
\begin{align}
(\rot\boldsymbol \rot\boldsymbol \tau, v)_K&=(\boldsymbol \tau, \boldsymbol \curl\curl v)_K+\sum_{e\in\mathcal E(K)}\sum_{\delta\in\partial e}\sign_{e,\delta}(\boldsymbol  n^{\intercal}\boldsymbol \tau\boldsymbol  t)(\delta)v(\delta) \notag\\
&\quad-\sum_{e\in\mathcal E(K)}\left[(\boldsymbol  t^{\intercal}\boldsymbol \tau \boldsymbol  t, \partial_n v)_{e}+(\partial_{t}(\boldsymbol  n^{\intercal}\boldsymbol \tau\boldsymbol  t)-\boldsymbol  t^{\intercal}\boldsymbol \rot\boldsymbol \tau,  v)_{e}\right] \label{eq:greenidentityrotrot}
\end{align}
for any $\boldsymbol \tau\in \mathcal C^2(K; \mathbb S)$ and $v\in H^2(K)$.

As a result, we have the following characterization of $\boldsymbol{H}(\rot \boldsymbol{\rot },\Omega; \mathbb{S})$.
\begin{lemma}\label{lem:Hrotrotpatching}
Let $\boldsymbol \tau\in \boldsymbol  L^2(\Omega;\mathbb S)$ such that
\begin{enumerate}[(i)]
\item $\boldsymbol \tau|_K\in \boldsymbol{H}(\rot \boldsymbol{\rot },K; \mathbb{S})$ for each $K\in\mathcal T_h$;

\smallskip
\item $(\boldsymbol  t^{\intercal}\boldsymbol \tau\boldsymbol  t)|_e\in L^2(e)$ is single-valued for each $e\in\mathcal E_h^i$;

\smallskip
\item $(-\partial_{t_e}(\boldsymbol  n^{\intercal}\boldsymbol \tau\boldsymbol  t)+\boldsymbol  t_e^{\intercal}\boldsymbol \rot\boldsymbol \tau)|_e\in L^2(e)$ is single-valued for each $e\in\mathcal E_h^i$;

\smallskip
\item $\boldsymbol \tau(\delta)$ is single-valued for each $\delta\in\mathcal V_h^i$,
\end{enumerate}
then $\boldsymbol \tau\in \boldsymbol{H}(\rot \boldsymbol{\rot },\Omega; \mathbb{S})$.
\end{lemma}

Take the space of shape functions
\[
\boldsymbol \Sigma_{\ell,k}^{\perp}(K):= \mathbb E_{\ell}(K;\mathbb S)\oplus\mathbb E_k^{\oplus}(K;\mathbb S)
\]
with $k\geq 3$ and $\ell\geq k-1$, where recall that the spaces $\mathbb E_{\ell}(K;\mathbb S), \mathbb E_k^{\oplus}(K;\mathbb S)$ are introduced in Remark \ref{rm:Ek}.
The local degrees of freedom are given by
\begin{align}
\boldsymbol \tau (\delta) & \quad\forall~\delta\in \mathcal V(K), \label{Hrotrotfemdof1}\\
(\boldsymbol  t^{\intercal}\boldsymbol \tau\boldsymbol  t, q)_e & \quad\forall~q\in\mathbb P_{\ell-2}(e),  e\in\mathcal E(K),\label{Hrotrotfemdof2}\\
(-\partial_{t}(\boldsymbol  n^{\intercal}\boldsymbol \tau\boldsymbol  t)+\boldsymbol  t^{\intercal}\boldsymbol \rot\boldsymbol \tau, q)_e & \quad\forall~q\in\mathbb P_{\ell-1}(e),  e\in\mathcal E(K),\label{Hrotrotfemdof3}\\
(\boldsymbol \tau, \boldsymbol \varsigma)_K & \quad\forall~\boldsymbol \varsigma\in\boldsymbol \curl\curl \, \mathbb P_{k-2}(K)\oplus
 \sym (\bs x\otimes \mathbb P_{\ell-2}(K;\mathbb R^2)).\label{Hrotrotfemdof4}
\end{align}
The degrees of freedom \eqref{Hrotrotfemdof1}-\eqref{Hrotrotfemdof4} are unisolvent for $\boldsymbol \Sigma_{\ell,k}^{\perp}(K)$.

Let
$
\mathbb A:=\begin{pmatrix}
0 & -1\\
1 & 0
\end{pmatrix}
$, then
\[
\boldsymbol  t=\mathbb A \boldsymbol  n,\quad \curl\phi=\mathbb A^{\intercal}\grad\phi,\quad \rot\boldsymbol  v=\div(\mathbb A^{\intercal} \boldsymbol  v), \quad \boldsymbol \rot\boldsymbol \tau=\mathbb A\boldsymbol \div(\mathbb A^{\intercal}\boldsymbol \tau\mathbb A),
\]
\[
 \rot\boldsymbol \rot\boldsymbol \tau=\div\boldsymbol \div(\mathbb A^{\intercal}\boldsymbol \tau\mathbb A),\;\; \boldsymbol \curl\curl v=\mathbb A^{\intercal}\nabla^2v\mathbb A,\;\; \boldsymbol \defm\boldsymbol  v=\mathbb A\sym\boldsymbol \curl(\mathbb A^{\intercal}\boldsymbol  v)\mathbb A^{\intercal},
\]
\[
-\partial_{t}(\boldsymbol  n^{\intercal}\boldsymbol \tau\boldsymbol  t)+\boldsymbol  t^{\intercal}\boldsymbol \rot\boldsymbol \tau=\partial_{t}(\boldsymbol  t^{\intercal}\mathbb A^{\intercal}\boldsymbol \tau\mathbb A\boldsymbol  n)+\boldsymbol  n^{\intercal}\boldsymbol \div(\mathbb A^{\intercal}\boldsymbol \tau\mathbb A)
\]
for sufficiently smooth scalar field $\phi$, vectorial field $\boldsymbol  v$ and tensorial field $\boldsymbol \tau$. Moreover, we have
\[
\boldsymbol  x^{\perp}=\mathbb A^{\intercal} \boldsymbol  x,\quad \mathbb E_{\ell}(K;\mathbb S)=\mathbb A^{\intercal}\mathbb C_{\ell}(K;\mathbb S)\mathbb A,\quad \mathbb E_k^{\oplus}(K;\mathbb S)=\mathbb A^{\intercal}\mathbb C_k^{\oplus}(K;\mathbb S)\mathbb A.
\]
Then the exactness of the complex \eqref{eq:divdivcomplexPolyvar} implies that the local finite element strain complex
\begin{equation}\label{eq:rotrotcomplexPolyvar}
\boldsymbol  {RM}\autorightarrow{$\subset$}{} \boldsymbol V_{\ell+1}(K)\autorightarrow{$\boldsymbol\defm$}{} \boldsymbol \Sigma_{\ell,k}^{\perp}(K) \autorightarrow{$\rot\boldsymbol{\rot}$}{} \mathbb P_{k-2}(K)\autorightarrow{}{}0
\end{equation}
is exact.

Define interpolation operators $\boldsymbol \Pi_K^{\perp}: \boldsymbol{H}^2(K; \mathbb{S})\to\boldsymbol \Sigma_{\ell,k}^{\perp}(K)$ and $\boldsymbol  I_K^{\perp}: \boldsymbol H^3(K;\mathbb R^2)\to\boldsymbol V_{\ell+1}(K)$ as
\[
\boldsymbol \Pi_K^{\perp}\boldsymbol \tau:=\mathbb A(\boldsymbol \Pi_K(\mathbb A^{\intercal}\boldsymbol \tau\mathbb A))\mathbb A^{\intercal}\quad\forall~\boldsymbol \tau\in\boldsymbol{H}^2(K; \mathbb{S}),
\]
\[
\boldsymbol  I_K^{\perp}\boldsymbol  v:=\mathbb A\boldsymbol  I_K(\mathbb A^{\intercal}\boldsymbol  v)\quad\forall~\boldsymbol  v\in \boldsymbol H^3(K;\mathbb R^2).
\]
It follows from \eqref{eq:PiKcd} and \eqref{eq:IKcd} that
\[
\rot\boldsymbol \rot\boldsymbol \Pi_K^{\perp}\boldsymbol \tau=\div\boldsymbol \div(\boldsymbol \Pi_K(\mathbb A^{\intercal}\boldsymbol \tau\mathbb A))=Q_{k-2}^K\div\boldsymbol \div(\mathbb A^{\intercal}\boldsymbol \tau\mathbb A)=Q_{k-2}^K\rot\boldsymbol \rot\boldsymbol \tau,
\]
\begin{align*}
\boldsymbol\defm(\boldsymbol  I_K^{\perp}\boldsymbol  v)&=\mathbb A\sym\curl(\mathbb A^{\intercal}\boldsymbol  I_K^{\perp}\boldsymbol  v)\mathbb A^{\intercal}=\mathbb A\sym\curl(\boldsymbol  I_K(\mathbb A^{\intercal}\boldsymbol  v))\mathbb A^{\intercal} \\
&=\mathbb A\boldsymbol \Pi_K(\sym\boldsymbol \curl(\mathbb A^{\intercal}\boldsymbol  v))\mathbb A^{\intercal} =\mathbb A\boldsymbol \Pi_K(\mathbb A^{\intercal}(\boldsymbol\defm\boldsymbol  v)\mathbb A)\mathbb A^{\intercal}=\boldsymbol \Pi_K^{\perp}(\boldsymbol\defm\boldsymbol  v).
\end{align*}
Therefore we have the following commutative diagram for the local finite element complex \eqref{eq:rotrotcomplexPolyvar}
$$
\begin{array}{c}
\xymatrix{
  \boldsymbol{RM} \ar[r]^-{\subset} & \boldsymbol  H^3(K;\mathbb R^2) \ar[d]^{\boldsymbol{I}_K^{\perp}} \ar[r]^-{\boldsymbol\defm}
                & \boldsymbol{H}^2(K; \mathbb{S}) \ar[d]^{\boldsymbol{\Pi}_K^{\perp}}   \ar[r]^-{\rot\boldsymbol{\rot}} & \ar[d]^{Q_K}L^2(K) \ar[r]^{} & 0 \\
 \boldsymbol{RM} \ar[r]^-{\subset} & \boldsymbol V_{\ell+1}(K) \ar[r]^{\boldsymbol\defm}
                &  \boldsymbol \Sigma_{\ell,k}^{\perp}(K)   \ar[r]^{\rot\boldsymbol{\rot}} &  \mathbb P_{k-2}(K) \ar[r]^{}& 0    }
\end{array}.
$$

Define
\begin{align*}
\boldsymbol\Sigma_h^{\perp}:=\{\boldsymbol \tau_h\in \boldsymbol  L^2(\Omega;\mathbb S):&\, \boldsymbol \tau_h|_K\in \boldsymbol \Sigma_{\ell,k}^{\perp}(K) \textrm{ for each } K\in\mathcal T_h, \textrm{ all the degrees of }\\
& \qquad\qquad\qquad\;\textrm{ freedom } \eqref{Hrotrotfemdof1}-\eqref{Hrotrotfemdof3} \textrm{ are single-valued}\}.
\end{align*}
Thanks to Lemma~\ref{lem:Hrotrotpatching}, the finite element space $\boldsymbol \Sigma_h^{\perp}\subset\boldsymbol{H}(\rot\boldsymbol{\rot},\Omega; \mathbb{S})$.
Let $\boldsymbol  I_h^{\perp}: \boldsymbol{H}^3(\Omega; \mathbb{S})\to\boldsymbol V_h$ and $\boldsymbol \Pi_h^{\perp}: \boldsymbol{H}^2(\Omega; \mathbb{S})\to\boldsymbol\Sigma_h^{\perp}$ be defined by $(\boldsymbol  I_h^{\perp}\boldsymbol  v)|_K:=\boldsymbol  I_K^{\perp}(\boldsymbol  v|_K)$ and $(\boldsymbol \Pi_h^{\perp}\boldsymbol \tau)|_K:=\boldsymbol \Pi_K^{\perp}(\boldsymbol \tau|_K)$, respectively.
Similarly as the commutative diagram \eqref{eq:divdivcdfem}, we have the commutative diagram
\begin{equation*}
\begin{array}{c}
\xymatrix{
  \boldsymbol{RM} \ar[r]^-{\subset} & \boldsymbol  H^3(\Omega;\mathbb R^2) \ar[d]^{\boldsymbol{I}_h^{\perp}} \ar[r]^-{\boldsymbol\defm}
                & \boldsymbol{H}^2(\Omega; \mathbb{S}) \ar[d]^{\boldsymbol{\Pi}_h^{\perp}}   \ar[r]^-{\rot\boldsymbol{\rot}} & \ar[d]^{Q_{h}}L^2(\Omega) \ar[r]^{} & 0 \\
 \boldsymbol{RM} \ar[r]^-{\subset} & \boldsymbol V_{h} \ar[r]^{\boldsymbol\defm}
                &  \boldsymbol \Sigma_{h}^{\perp}   \ar[r]^{\rot\boldsymbol{\rot}} &  \mathcal Q_h \ar[r]^{}& 0    }
\end{array}.
\end{equation*}

\section{Mixed finite element methods for biharmonic equation}

In this section we will apply the $\boldsymbol H(\div\boldsymbol\div)$-conforming finite element pair $(\boldsymbol\Sigma_h, \mathcal Q_h)$ to solve the biharmonic equation
\begin{equation}\label{eq:biharmonic}
\begin{cases}
\quad\;\;\,\Delta^2u = -f & \textrm{ in } \Omega,\\
u=\partial_nu=0 & \textrm{ on } \partial\Omega,
\end{cases}
\end{equation}
where $f\in L^2(\Omega)$. A mixed formulation of the biharmonic equation \eqref{eq:biharmonic} is to find $\boldsymbol \sigma\in\boldsymbol{H}(\div \boldsymbol{\div },\Omega; \mathbb{S})$ and $u\in L^2(\Omega)$ such that
\begin{align}
(\boldsymbol{\sigma}, \boldsymbol{\tau})+(\div\boldsymbol \div\boldsymbol{\tau}, u)&=0 \quad\quad\quad\;\; \forall~\boldsymbol{\tau}\in\boldsymbol{H}(\div \boldsymbol{\div},\Omega; \mathbb{S}), \label{mixedform1} \\
(\div\boldsymbol \div\boldsymbol{\sigma}, v)&=(f, v) \quad\;\;\; \forall~v\in L^2(\Omega). \label{mixedform2}
\end{align}
Note that Dirichlet-type boundary of $u$ is imposed as natural condition in the mixed formulation.

\subsection{Mixed finite element methods}
Employing the finite element spaces $\boldsymbol\Sigma_h\times\mathcal Q_h$ to discretize $\boldsymbol{H}(\div \boldsymbol{\div },\Omega; \mathbb{S})\times L^2(\Omega)$, we propose the following discrete methods for the mixed formulation \eqref{mixedform1}-\eqref{mixedform2}:
find $\boldsymbol \sigma_h\in\boldsymbol\Sigma_h$ and $u_h\in \mathcal Q_h$ such that
\begin{align}
(\boldsymbol{\sigma}_h, \boldsymbol{\tau}_h)+(\div\boldsymbol \div\boldsymbol{\tau}_h, u_h)&=0 \quad\quad\quad\;\; \forall~\boldsymbol{\tau}_h\in \boldsymbol\Sigma_h, \label{mfem1} \\
(\div\boldsymbol \div\boldsymbol{\sigma}_h, v_h)&=(f, v_h) \quad\; \forall~v_h\in \mathcal Q_h. \label{mfem2}
\end{align}

As a result of \eqref{eq:Pihcd} and \eqref{eq:PiKestimate}, we have the inf-sup condition
\begin{equation*}
\|v_h\|_0\lesssim \sup_{\boldsymbol{\tau}_h\in\boldsymbol\Sigma_h}\frac{(\div\boldsymbol \div\boldsymbol{\tau}_h, v_h)}{\|\boldsymbol{\tau}_h\|_{\boldsymbol{H}(\div \boldsymbol{\div })}}.
\end{equation*}
By the Babu{\v{s}}ka-Brezzi theory \cite{BoffiBrezziFortin2013}, the following stability result holds
\begin{equation}\label{eq:discretestability}
\|\widetilde{\boldsymbol{\sigma}}_h\|_{\boldsymbol{H}(\div \boldsymbol{\div })}+\|\widetilde{u}_h\|_0\lesssim \sup_{\boldsymbol{\tau}_h\in\boldsymbol\Sigma_h\atop v_h\in\mathcal Q_h}\frac{(\widetilde{\boldsymbol{\sigma}}_h, \boldsymbol{\tau}_h)+(\div\boldsymbol \div\boldsymbol{\tau}_h, \widetilde{u}_h)+(\div\boldsymbol \div\widetilde{\boldsymbol{\sigma}}_h, v_h)}{\|\boldsymbol{\tau}_h\|_{\boldsymbol{H}(\div \boldsymbol{\div })} + \|v_h\|_0}
\end{equation}
for any $\widetilde{\boldsymbol{\sigma}}_h\in\boldsymbol\Sigma_h$ and $\widetilde{u}_h\in\mathcal Q_h$.
Hence the mixed finite element method \eqref{mfem1}-\eqref{mfem2} is well-posed.

\begin{theorem}
Let $\boldsymbol \sigma_h\in\boldsymbol\Sigma_h$ and $u_h\in \mathcal Q_h$ be the solution of the mixed finite element methods \eqref{mfem1}-\eqref{mfem2}. Assume $\boldsymbol \sigma\in\boldsymbol{H}^{\min\{\ell,k\}+1}(\Omega; \mathbb{S})$, $u\in H^{k-1}(\Omega)$ and $f\in H^{k-1}(\Omega)$. Then
\begin{align}
\label{eq:errorestimate1}
\|\boldsymbol{\sigma}-\boldsymbol{\sigma}_h\|_0+\|Q_h u-u_h\|_0\lesssim h^{\min\{\ell,k\}+1}|\boldsymbol \sigma|_{\min\{\ell,k\}+1},
\\
\label{eq:errorestimate2}
\|u-u_h\|_0\lesssim h^{\min\{\ell,k\}+1}|\boldsymbol \sigma|_{\min\{\ell,k\}+1}+h^{k-1}|u|_{k-1},
\\
\label{eq:errorestimate3}
\|\boldsymbol{\sigma}-\boldsymbol{\sigma}_h\|_{\boldsymbol{H}(\div \boldsymbol{\div })}\lesssim h^{\min\{\ell,k\}+1}|\boldsymbol \sigma|_{\min\{\ell,k\}+1}+h^{k-1}|f|_{k-1}.
\end{align}
\end{theorem}
\begin{proof}
Subtracting \eqref{mfem1}-\eqref{mfem2} from \eqref{mixedform1}-\eqref{mixedform2}, it follows
\[
(\boldsymbol{\sigma}-\boldsymbol{\sigma}_h, \boldsymbol{\tau}_h)+(\div\boldsymbol \div\boldsymbol{\tau}_h, u-u_h)+(\div\boldsymbol \div(\boldsymbol{\sigma}-\boldsymbol{\sigma}_h), v_h)=0,
\]
which combined with \eqref{eq:Pihcd} yields
\begin{equation}\label{eq:errorequation}
(\boldsymbol{\sigma}-\boldsymbol{\sigma}_h, \boldsymbol{\tau}_h)+(\div\boldsymbol \div\boldsymbol{\tau}_h, Q_h u-u_h)+(\div\boldsymbol \div(\boldsymbol \Pi_h\boldsymbol{\sigma}-\boldsymbol{\sigma}_h), v_h)=0.
\end{equation}
Taking $\widetilde{\boldsymbol{\sigma}}_h=\boldsymbol \Pi_h\boldsymbol{\sigma}-\boldsymbol{\sigma}_h$ and $\widetilde{u}_h=Q_h u-u_h$ in \eqref{eq:discretestability}, we get
\begin{align*}
\|\boldsymbol \Pi_h\boldsymbol{\sigma}-\boldsymbol{\sigma}_h\|_{\boldsymbol{H}(\div \boldsymbol{\div })}+\|Q_h u-u_h\|_0&\lesssim \sup_{\boldsymbol{\tau}_h\in\boldsymbol\Sigma_h\atop v_h\in\mathcal Q_h}\frac{(\boldsymbol \Pi_h\boldsymbol{\sigma}-\boldsymbol{\sigma}, \boldsymbol{\tau}_h)}{\|\boldsymbol{\tau}_h\|_{\boldsymbol{H}(\div \boldsymbol{\div })} + \|v_h\|_0} \\
&\leq \|\boldsymbol \Pi_h\boldsymbol{\sigma}-\boldsymbol{\sigma}\|_0.
\end{align*}
Hence we have
\begin{align*}
\|\boldsymbol{\sigma}-\boldsymbol{\sigma}_h\|_0+\|Q_h u-u_h\|_0\lesssim \|\boldsymbol{\sigma}-\boldsymbol \Pi_h\boldsymbol{\sigma}\|_0,
\\
\|u-u_h\|_0\lesssim \|\boldsymbol{\sigma}-\boldsymbol \Pi_h\boldsymbol{\sigma}\|_0+\|u-Q_h u\|_0,
\\
\|\boldsymbol{\sigma}-\boldsymbol{\sigma}_h\|_{\boldsymbol{H}(\div \boldsymbol{\div })}\lesssim \|\boldsymbol{\sigma}-\boldsymbol \Pi_h\boldsymbol{\sigma}\|_{\boldsymbol{H}(\div \boldsymbol{\div })}.
\end{align*}
Finally we conclude \eqref{eq:errorestimate1}, \eqref{eq:errorestimate2} and \eqref{eq:errorestimate3} from \eqref{eq:PiKestimate}.
\end{proof}

If we are interested in the approximation of stress in $\boldsymbol  H(\div\boldsymbol  \div)$ norm, it is more economic to chose $\ell = k - 1$.
If instead the $L^2$-norm is of concern, $\ell = k$ is a better choice to achieve higher accuracy.

The estimate of $\|Q_hu-u_h\|_0$ in \eqref{eq:errorestimate1} is superconvergent and can be used to postprocess to get a high order approximation of displacement.

\subsection{Superconvergence of displacement in mesh-dependent norm}

Equip the space $$H^2(\mathcal T_h):=\{v\in L^2(\Omega): v|_K\in H^2(K) \textrm{ for each } K\in\mathcal T_h\}$$ with squared mesh-dependent norm
\[
|v|_{2,h}^2:=\sum_{K\in\mathcal{T}_h}|v|_{2,K}^2+\sum_{e\in\mathcal E_h}\left(h_e^{-3}\|\llbracket v\rrbracket\|_{0,e}^2+h_e^{-1}\|\llbracket \partial_{n_e}v\rrbracket\|_{0,e}^2\right),
\]
where $\llbracket v\rrbracket$ and $\llbracket \partial_{n_e}v\rrbracket$ are jumps of $v$ and $\partial_{n_e}v$ across $e$ for $e\in \mathcal E_h^i$, and $\llbracket v\rrbracket=v$ and $\llbracket \partial_{n_e}v\rrbracket=\partial_{n_e}v$ for $e\in \mathcal E_h\backslash\mathcal E_h^i$.

\begin{lemma}
It holds the inf-sup condition
\begin{equation}\label{eq:discreteinfsupmd}
|v_h|_{2,h}\lesssim \sup_{\boldsymbol{\tau}_h\in\boldsymbol\Sigma_h}\frac{(\div\boldsymbol \div\boldsymbol{\tau}_h, v_h)}{\|\boldsymbol{\tau}_h\|_0}\quad\forall~v_h\in\mathcal Q_h.
\end{equation}
\end{lemma}
\begin{proof}
Let $\boldsymbol{\tau}_h\in\boldsymbol\Sigma_h$ be determined by
\begin{align*}
\boldsymbol \tau_h (\delta)&=0  \qquad\qquad\quad\;\;\;\forall~\delta\in \mathcal V_h, \\
Q_{\ell-2}^e(\boldsymbol  n^{\intercal}\boldsymbol \tau_h\boldsymbol n)&=-h_e^{-1}\llbracket \partial_{n_e}v_h\rrbracket  \;\;\forall~ e\in\mathcal E_h, \\
\partial_{t_e}(\boldsymbol  t^{\intercal}\boldsymbol \tau_h\boldsymbol  n)+\boldsymbol n_e^{\intercal}\boldsymbol \div\boldsymbol \tau_h&=h_e^{-3}\llbracket v_h\rrbracket  \qquad\;\;\;\forall~ e\in\mathcal E_h, \\
(\boldsymbol \tau_h, \boldsymbol \varsigma)_K&=(\nabla^2v_h, \boldsymbol \varsigma)_K  \quad\;\;\forall~\boldsymbol \varsigma\in\nabla^2\mathbb P_{k-2}(K), \\
(\boldsymbol \tau_h, \boldsymbol \varsigma)_K&=0  \qquad\qquad\quad\;\;\;\forall~\boldsymbol \varsigma\in\sym (\bs x^{\perp}\otimes \mathbb P_{\ell-2}(K;\mathbb R^2))
\end{align*}
for each $K\in\mathcal T_h$.
Due to the scaling argument, it holds
\[
\|\boldsymbol \tau_h\|_0\lesssim |v|_{2,h}.
\]
And we get from \eqref{eq:greenidentitydivdiv} that
\[
(\div\boldsymbol \div\boldsymbol \tau_h, v_h)=|v|_{2,h}^2.
\]
Hence the inf-sup condition \eqref{eq:discreteinfsupmd} follows.
\end{proof}

An immediate result of the inf-sup condition \eqref{eq:discreteinfsupmd} is the stability result
\begin{equation}\label{eq:discretestabilitymd}
\|\widetilde{\boldsymbol{\sigma}}_h\|_{0}+|\widetilde{u}_h|_{2,h}\lesssim \sup_{\boldsymbol{\tau}_h\in\boldsymbol\Sigma_h\atop v_h\in\mathcal Q_h}\frac{(\widetilde{\boldsymbol{\sigma}}_h, \boldsymbol{\tau}_h)+(\div\boldsymbol \div\boldsymbol{\tau}_h, \widetilde{u}_h)+(\div\boldsymbol \div\widetilde{\boldsymbol{\sigma}}_h, v_h)}{\|\boldsymbol{\tau}_h\|_{0} + |v_h|_{2,h}}
\end{equation}
for any $\widetilde{\boldsymbol{\sigma}}_h\in\boldsymbol\Sigma_h$ and $\widetilde{u}_h\in\mathcal Q_h$.

\begin{theorem}
Let $\boldsymbol \sigma_h\in\boldsymbol\Sigma_h$ and $u_h\in \mathcal Q_h$ be the solution of the mixed finite element methods \eqref{mfem1}-\eqref{mfem2}. Assume $\boldsymbol \sigma\in\boldsymbol{H}^{\min\{\ell,k\}+1}(\Omega; \mathbb{S})$. Then
\begin{equation}\label{eq:uH2superconvergence}
|Q_hu-u_h|_{2,h}\lesssim h^{\min\{\ell,k\}+1}|\boldsymbol \sigma|_{\min\{\ell,k\}+1}.
\end{equation}
\end{theorem}
\begin{proof}
Taking $\widetilde{\boldsymbol{\sigma}}_h=\boldsymbol \Pi_h\boldsymbol{\sigma}-\boldsymbol{\sigma}_h$ and $\widetilde{u}_h=Q_hu-u_h$ in \eqref{eq:discretestabilitymd}, we get from \eqref{eq:errorequation} that
\[
\|\boldsymbol \Pi_h\boldsymbol{\sigma}-\boldsymbol{\sigma}_h\|_{0}+|Q_hu-u_h|_{2,h}\lesssim \sup_{\boldsymbol{\tau}_h\in\boldsymbol\Sigma_h\atop v_h\in\mathcal Q_h}\frac{(\boldsymbol \Pi_h\boldsymbol{\sigma}-\boldsymbol{\sigma}, \boldsymbol{\tau}_h)}{\|\boldsymbol{\tau}_h\|_{0} + |v_h|_{2,h}}\leq \|\boldsymbol \Pi_h\boldsymbol{\sigma}-\boldsymbol{\sigma}\|_0,
\]
which gives \eqref{eq:uH2superconvergence}.
\end{proof}

The estimate of $|Q_hu-u_h|_{2,h}$ in \eqref{eq:uH2superconvergence} is superconvergent, which is  $\min\{\ell-k,0\}+4$ order higher than the optimal one and will be used to get a high order approximation of displacement by postprocessing.

\subsection{Postprocessing}
Define $u_h^{\ast}\in \mathbb P_{\min\{\ell,k\}+2}(\mathcal T_h)$ as follows: for each $K\in\mathcal T_h$,
\begin{align*}
(\nabla^2u_h^{\ast}, \nabla^2 q)_K&=-(\boldsymbol \sigma_h, \nabla^2 q)_K\quad\forall~q\in\mathbb P_{\min\{\ell,k\}+2}(\mathcal T_h),
\\
(u_h^{\ast}, q)_K&=(u_h, q)_K\qquad\quad\;\forall~q\in\mathbb P_{1}(\mathcal T_h).
\end{align*}
Namely we compute the projection of $\boldsymbol  \sigma_h$ in $H^2$ semi-inner product and use $u_h$ to impose the constraint. Recall that $k\geq 3$ and $\ell\geq k-1$. Thus $u_h\in \mathbb P_{k-2}(\mathcal T_h)$ and the local $H^2$-projection is well-defined.

\begin{theorem}
Let $\boldsymbol \sigma_h\in\boldsymbol\Sigma_h$ and $u_h\in \mathcal Q_h$ be the solution of the mixed finite element methods \eqref{mfem1}-\eqref{mfem2}. Assume $u\in H^{\min\{\ell,k\}+3}(\Omega)$. Then
\begin{equation}\label{eq:uastH2superconvergence}
|u-u_h^{\ast}|_{2,h}\lesssim h^{\min\{\ell,k\}+1}|u|_{\min\{\ell,k\}+3}.
\end{equation}
\end{theorem}
\begin{proof}
Let $z=(I-Q_h^1)(Q_h^{\min\{\ell,k\}+2}u-u_h^{\ast})$. By the definition of $u_h^{\ast}$, it follows
\begin{align*}
\|\nabla^2z\|_{0,K}^2&=(\nabla^2(Q_{\min\{\ell,k\}+2}^Ku-u), \nabla^2z)_{K}+(\nabla^2(u-u_h^{\ast}), \nabla^2z)_{K} \\
&=(\nabla^2(Q_{\min\{\ell,k\}+2}^Ku-u), \nabla^2z)_{K}+(\boldsymbol \sigma_h-\boldsymbol \sigma, \nabla^2z)_{K},
\end{align*}
which implies
\begin{equation}\label{eq:20200424-1}
|z|_{2,h}^2\lesssim \sum_{K\in\mathcal T_h}|u-Q_h^{\min\{\ell,k\}+2}u|_{2,K}^2+\|\boldsymbol \sigma-\boldsymbol \sigma_h\|_{0}^2.
\end{equation}
Since $Q_h^1(Q_h^{\min\{\ell,k\}+2}u-u_h^{\ast})=Q_h^1(Q_h^{k-2}u-u_h)$ and
\[
|(I-Q_h^1)(Q_h^{k-2}u-u_h)|_{2,h}^2\lesssim \sum_{K\in\mathcal T_h}|Q_h^{k-2}u-u_h|_{2,K}^2,
\]
we get
\begin{equation}\label{eq:20200424-2}
|Q_h^1(Q_h^{\min\{\ell,k\}+2}u-u_h^{\ast})|_{2,h}\lesssim |Q_h^1(Q_h^{k-2}u-u_h)|_{2,h}\lesssim |Q_h^{k-2}u-u_h|_{2,h}.
\end{equation}
Combining \eqref{eq:20200424-1} and \eqref{eq:20200424-2} gives
\[
|Q_h^{\min\{\ell,k\}+2}u-u_h^{\ast}|_{2,h}\leq |u-Q_h^{\min\{\ell,k\}+2}u|_{2,h}+\|\boldsymbol \sigma-\boldsymbol \sigma_h\|_{0}+|Q_h^{k-2}u-u_h|_{2,h}.
\]
Finally \eqref{eq:uastH2superconvergence} follows from \eqref{eq:errorestimate1} and \eqref{eq:uH2superconvergence}.
\end{proof}

\subsection{Hybridization}

In this subsection we consider a partial hybridization of the mixed finite element methods \eqref{mfem1}-\eqref{mfem2} by relaxing the continuity of the effective transverse shear force.
To this end, let
\begin{align*}
\widetilde{\boldsymbol\Sigma}_h:=\{\boldsymbol \tau_h\in \boldsymbol  L^2(\Omega;\mathbb S):&\, \boldsymbol \tau_h|_K\in \boldsymbol \Sigma_{\ell,k}(K) \textrm{ for each } K\in\mathcal T_h, \\
& \textrm{ all the degrees of freedom } \eqref{Hdivdivfemdof1}-\eqref{Hdivdivfemdof2} \textrm{ are single-valued}\},
\end{align*}
\begin{align*}
\Lambda_h:=\{\mu_h\in L^2(\mathcal E_h):& \, \mu_h|_e\in\mathbb P_{\ell-1}(e) \textrm{ for each } e\in\mathcal E_h^i,\\
 & \,\textrm{and } \mu_h|_e=0 \textrm{ for each } e\in\mathcal E_h\backslash\mathcal E_h^i \}.
\end{align*}

\begin{lemma}
Let $\boldsymbol \sigma_h\in\boldsymbol\Sigma_h$ and $u_h\in \mathcal Q_h$ be the solution of the mixed finite element methods \eqref{mfem1}-\eqref{mfem2}.
Let $(\widetilde{\boldsymbol\sigma}_h, \widetilde{u}_h, \lambda_h)\in\widetilde{\boldsymbol\Sigma}_h\times\mathcal Q_h\times\Lambda_h$ satisfy the partial hybridized mixed finite methods
\begin{align}
(\widetilde{\boldsymbol\sigma}_h, \boldsymbol{\tau}_h)+b_h(\boldsymbol{\tau}_h, \widetilde{u}_h, \lambda_h)&=0 \quad\quad\quad\;\; \forall~\boldsymbol{\tau}_h\in \widetilde{\boldsymbol\Sigma}_h, \label{mfemhybrid1} \\
b_h(\widetilde{\boldsymbol\sigma}_h, v_h, \mu_h)&=(f, v_h) \quad\; \forall~v_h\in \mathcal Q_h,\; \mu_h\in\Lambda_h, \label{mfemhybrid2}
\end{align}
where
\[
b_h(\boldsymbol{\tau}_h, v_h, \mu_h):=\sum_{K\in\mathcal T_h}(\div\boldsymbol \div\boldsymbol \tau_h, v_h)_K-\sum_{K\in\mathcal T_h}(\partial_{t}(\boldsymbol  t^{\intercal}\boldsymbol \tau_h\boldsymbol  n)+\boldsymbol  n^{\intercal}\boldsymbol \div\boldsymbol \tau_h,  \mu_h)_{\partial K}.
\]
Then $\widetilde{\boldsymbol\sigma}_h=\boldsymbol \sigma_h$ and $\widetilde{u}_h=u_h$.
\end{lemma}
\begin{proof}
First show the unisolvence of the discrete methods \eqref{mfemhybrid1}-\eqref{mfemhybrid2}. Assume $f$ is zero.
Due to \eqref{mfemhybrid2} with $v_h=0$, we get $\widetilde{\boldsymbol\sigma}_h\in\boldsymbol\Sigma_h$.
Hence $(\widetilde{\boldsymbol\sigma}_h, \widetilde{u}_h)\in\boldsymbol\Sigma_h\times\mathcal Q_h$ satisfies the mixed finite element methods \eqref{mfem1}-\eqref{mfem2} with $f=0$. Then $\widetilde{\boldsymbol\sigma}_h=\boldsymbol0$ and $\widetilde{u}_h=0$ follows from the unisolvence of the mixed methods \eqref{mfem1}-\eqref{mfem2}.  Thus \eqref{mfemhybrid1} becomes
\[
\sum_{K\in\mathcal T_h}(\partial_{t}(\boldsymbol  t^{\intercal}\boldsymbol \tau_h\boldsymbol  n)+\boldsymbol  n^{\intercal}\boldsymbol \div\boldsymbol \tau_h,  \lambda_h)_{\partial K}=0\quad\forall~\boldsymbol\tau_h\in\boldsymbol\Sigma_h.
\]
Now taking $\boldsymbol{\tau}_h\in\boldsymbol\Sigma_h$ such that
\begin{align*}
\boldsymbol \tau_h (\delta)&=0  \quad\;\forall~\delta\in \mathcal V_h, \\
Q_{\ell-2}^e(\boldsymbol  n^{\intercal}\boldsymbol \tau_h\boldsymbol n)&=0  \;\;\;\;\;\forall~ e\in\mathcal E_h, \\
\partial_{t_e}(\boldsymbol  t^{\intercal}\boldsymbol \tau_h\boldsymbol  n)+\boldsymbol n_e^{\intercal}\boldsymbol \div\boldsymbol \tau_h&=\lambda_h  \;\;\,\forall~ e\in\mathcal E_h, \\
(\boldsymbol \tau_h, \boldsymbol \varsigma)_K&=0  \quad\;\forall~\boldsymbol \varsigma\in\nabla^2\mathbb P_{k-2}(K)\oplus\sym (\bs x^{\perp}\otimes \mathbb P_{\ell-2}(K;\mathbb R^2))
\end{align*}
for each $K\in\mathcal T_h$, we acquire $\lambda_h=0$.

For general $f\in L^2(\Omega)$, it also follows from \eqref{mfemhybrid2} that $\widetilde{\boldsymbol\sigma}_h\in\boldsymbol\Sigma_h$. And then $(\widetilde{\boldsymbol\sigma}_h, \widetilde{u}_h)\in\boldsymbol\Sigma_h\times\mathcal Q_h$ satisfies the mixed finite element methods \eqref{mfem1}-\eqref{mfem2}. Thus $\widetilde{\boldsymbol\sigma}_h=\boldsymbol \sigma_h$ and $\widetilde{u}_h=u_h$.
\end{proof}

The space of shape functions for $\widetilde{\boldsymbol\Sigma}_h$ is still $\boldsymbol \Sigma_{\ell,k}(K)$.
The local degrees of freedom are 
\begin{align*}
\boldsymbol \tau (\delta) & \quad\forall~\delta\in \mathcal V(K), \\
(\boldsymbol  n^{\intercal}\boldsymbol \tau\boldsymbol  n, q)_e & \quad\forall~q\in\mathbb P_{\ell-2}(e),  e\in\mathcal E(K), \\
(\partial_{t}(\boldsymbol  t^{\intercal}\boldsymbol \tau\boldsymbol  n)+\boldsymbol  n^{\intercal}\boldsymbol \div\boldsymbol \tau, q)_e & \quad\forall~q\in\mathbb P_{\ell-1}(e),  e\in\mathring{\mathcal E}(K), \\
(\boldsymbol \tau, \boldsymbol \varsigma)_K & \quad\forall~\boldsymbol \varsigma\in\nabla^2\mathbb P_{k-2}(K)\oplus \sym (\bs x^{\perp}\otimes \mathbb P_{\ell-2}(K;\mathbb R^2)). 
\end{align*}
Here notation $e\in\mathring{\mathcal E}(K)$ means $(\partial_{t}(\boldsymbol  t^{\intercal}\boldsymbol \tau\boldsymbol  n)+\boldsymbol  n^{\intercal}\boldsymbol \div\boldsymbol \tau, q)_e$ are the interior degrees of freedom, i.e., $(\partial_{t}(\boldsymbol  t^{\intercal}\boldsymbol \tau\boldsymbol  n)+\boldsymbol  n^{\intercal}\boldsymbol \div\boldsymbol \tau, q)_e$ are double-valued on each edge $e\in\mathcal E_h^i$.

When $\ell=k$, we can take the following degrees of freedom
\begin{align*}
\boldsymbol \tau (\delta) & \quad\forall~\delta\in \mathcal V(K), \\
(\boldsymbol  n^{\intercal}\boldsymbol \tau\boldsymbol  n, q)_e & \quad\forall~q\in\mathbb P_{k-2}(e),  e\in\mathcal E(K), \\
(\boldsymbol  t^{\intercal}\boldsymbol \tau\boldsymbol  n, q)_e,\; (\boldsymbol  t^{\intercal}\boldsymbol \tau\boldsymbol  t, q)_e & \quad\forall~q\in\mathbb P_{k-2}(e),  e\in\mathring{\mathcal E}(K), \\
(\boldsymbol \tau, \boldsymbol \varsigma)_K & \quad\forall~\boldsymbol \varsigma\in \mathbb P_{k-3}(K;\mathbb S). 
\end{align*}
These are exactly the tensor version of the local degrees of freedom for the Lagrange element.
Therefore we can adopt the standard Lagrange element basis to implement the hybridized mixed finite element methods \eqref{mfemhybrid1}-\eqref{mfemhybrid2}.

\bibliographystyle{abbrv}
\bibliography{./paper}
\end{document}